\documentclass[11pt]{article}
\usepackage[dvipdfmx]{graphicx}%図，画像用
\usepackage[dvipdfmx]{color}%図，画像用
\usepackage{amsmath,amsthm,amssymb,amscd,ascmac,amsfonts,multicol,tabularx,url,stmaryrd, cases}

\usepackage{here}%図を出したいところに出す

\makeatletter
\@addtoreset{equation}{section}

\makeatother

\newtheorem{main thm}{Theorem}
\newtheorem{df}{\quad Definition}[section]
\newtheorem{thm}[df]{Theorem}
\newtheorem{lem}[df]{Lemma}
\newtheorem{cor}[df]{Corollary}
\newtheorem{prop}[df]{Proposition}
\newtheorem{rem}[df]{Remark}

\newcommand{\Z}{\mathbb{Z}}
\newcommand{\Q}{\mathbb{Q}}

\newcommand{\Zp}{\mathbb{Z}_p}
\newcommand{\Qp}{\mathbb{Q}_p}

\newcommand{\Fp}{\mathbb{F}_p}

\newcommand{\p}{\mathfrak{p}}

\renewcommand{\P}{\mathfrak{P}}

\newcommand{\alp}{\alpha}

\newcommand{\gam}{\gamma}

\newcommand{\del}{\delta}

\newcommand{\lam}{\lambda}
\newcommand{\Lam}{\Lambda}

\newcommand{\Gal}{{\rm Gal}}

\newcommand{\Ker}{{\rm Ker}}
\newcommand{\Coker}{{\rm Coker}}

\newcommand{\ord}{{\rm ord}}

\newcommand{\Ki}{K_{\infty}}

\newcommand{\isom}{\simeq}

\newcommand{\x}{\times}
\newcommand{\op}{\oplus}
\newcommand{\bigop}{\bigoplus}
\newcommand{\ox}{\otimes}
\newcommand{\inj}{\hookrightarrow}
\newcommand{\surj}{\twoheadrightarrow}

\newcommand{\up}{\stackrel}

\newcommand{\f}[2]{\dfrac{#1}{#2}}

\renewcommand{\tilde}{\widetilde}

\DeclareMathOperator*{\zetaprod}{{\displaystyle{\prod\kern-1.45em\coprod}}}

\begin{document}
\begin{center}
{\Large {Galois Coinvariants of the Unramified Iwasawa Modules 
\\
of Multiple $\mathbb{Z}_p$-Extensions}
}\\
\medskip{Takashi MIURA, Kazuaki MURAKAMI, Rei OTSUKI, and Keiji OKANO}
\footnote[0]{2010 \textit{Mathematics Subject Classification}. 
%Primary 11T71, Secondary 14H52.} 
11R23.} 
\end{center}

%%%%%%%%%%%%%%%%%%%%%%%%%%%%%%%%%%%%%%%%%%%%%%%%%%%%%%%%%%

\begin{abstract}
For a CM-field $K$ and an odd prime number $p$, 
let $\widetilde K'$ be a certain multiple $\Z_p$-extension of $K$.
In this paper, we study several basic properties of the unramified Iwasawa module $X_{\widetilde K'}$ of $\widetilde K'$
as a $\Z_p[[\Gal(\widetilde K'/K)]]$-module.
Our first main result is a description of 
the order of a Galois coinvariant of $X_{\widetilde K'}$ in terms of 
the characteristic power series of the unramified Iwasawa module of the cyclotomic $\Z_p$-extension of $K$ 
under a certain assumption on the splitting of primes above $p$. 
Second one is that if $K$ is an imaginary quadratic field and $p$ does not split in $K$, 
we give a necessary and sufficient condition for which $X_{\widetilde K}$ is 
$\Z_p[[\Gal(\widetilde K/K)]]$-cyclic
under several assumptions on the Iwasawa $\lambda$-invariant and the ideal class group of $K$, 
where $\widetilde K$ is the $\Z_p^2$-extension of $K$.
\end{abstract}

%%%%%%%%%%%%%%%%%%%%%%%%%%%%%%%%%%%%%%%%%%%%%%%%%%%%%%%%%%

\section{Introduction}\label{intro}
\subsection{The unramified Iwasawa modules}

Let $p$ be an arbitrary prime number, $K$ a finite extension of the rational number field $\Q$ and $\Ki/K$ the cyclotomic $\Zp$-extension.
For an arbitrary algebraic number field $F$, we denote by $X_F$ the Galois group of the maximal unramified abelian $p$-extension over $F$.
The module $X_{\Ki}$ is called the unramified Iwasawa module of $\Ki$.
In the case where $K$ is totally real, little is known about the structure of $X_{\Ki}$, although the Iwasawa main conjecture gives us highly nontrivial 
%mysterious 
information about 
the minus part of $X_{\Ki}$ in the case where $K$ is a CM-field.
%the Galois group of the maximal $p$-extension of $\Ki$ which is unramified outside all primes above  $p$.
Greenberg conjectured in \cite{Greenberg76} that $X_{\Ki}$ would be finite if $K$ is totally real, which is called Greenberg's conjecture. 
A lot of efforts by a number of mathematicians have revealed that this conjecture holds true in many cases, 
but it still remains unsolved (in general).
%but it has been still remaining unsolved.
%
%It has been proved in many cases by many mathematicians, but still remains unsolved.
%It is still unsolved, although many mathematicians prove in many cases.

We consider the maximal multiple $\Zp$-extension $\tilde{K}/K$ and its unramified Iwasawa module $X_{\tilde{K}}$.
It is known that $X_{\tilde{K}}$ is a finitely generated torsion $\Zp[[\Gal(\tilde{K}/K)]]$-module (see \cite{Greenberg73}).
%Minardi conjectured in \cite{Minardi} 
There is a conjecture that $X_{\tilde{K}}$ would be pseudo-null as a $\Zp[[\Gal(\tilde{K}/K)]]$-module, which is called Greenberg's generalized conjecture
(``pseudo-null'' is defined 
in \S \ref{Proof of Thm 2}). 
Concerning this conjecture and its application, there are many studies 
(Bleher et al. \cite{BCGKPST}, Fujii \cite{Fujii17}, Itoh \cite{Itoh11}, Ozaki \cite{Ozaki01}, and Minardi \cite{Minardi}, etc.).
However, even if Greenberg's generalized conjecture is true, it just states that the characteristic ideal of $X_{\tilde{K}}$ is trivial, 
so that it seems difficult to consider any
%some 
analogues of the Iwasawa invariants and the Iwasawa main conjecture for $X_{\tilde{K}}$.
Therefore, it is worthwhile to study not only Greenberg's generalized conjecture, but also various basic properties of $X_{\tilde{K}}$, 
for example, the number of generators as a $\Zp[[\Gal(\tilde{K}/K)]]$-module, its Galois (co)invariants, and cohomogical properties.

In the following, we always assume that $p$ is odd.
In this paper, we study the Galois coinvariants of the unramified Iwasawa modules of a certain multiple $\Z_p$-extension in a relatively general situation.
Roughly speaking, this paper consists of two parts 
``{\it Split case}'' (\S \ref{Proof of Theorem 1}, \ref{Proof of Thm 2}) 
and 
``{\it Non-split case}'' (\S \ref{Proof of Thm 3}, \ref{Classification in the case lam=2}).
In {\it Split case}, we consider a certain multiple $\Z_p$-extension $\widetilde{K}'$ of a CM-field $K$ which satisfies the condition of Gross's conjecture of rank one.
If $K$ is an abelian extension in which $p$ splits completely and the degree of $K$ is coprime to $p$,
then $\widetilde{K}'$ is coincide with $\tilde{K}$.
Our first main result is a description of 
the order of a Galois coinvariant of $X_{\widetilde K'}$ in terms of 
the characteristic power series of $X_{\Ki}$
(Theorem \ref{thm 1}).
In addition, for an imaginary quadratic field $K$ in which $p$ splits as $(p)=\P \overline{\P}$, we give a theorem which suggests that the characteristic ideal of $\P$-ramified Iwasawa module of $\tilde{K}$ relates to the structure of $X_{\tilde{K}}$ (Theorem \ref{thm 2}).
On the other hand, in {\it Non-split case}, we consider $\widetilde K$ for an imaginary quadratic field $K$ in which $p$ does not split.
Our second main result is to give a necessary and sufficient condition for which $X_{\widetilde K}$ is 
$\Z_p[[\Gal(\widetilde K/K)]]$-cyclic
under several assumptions on the Iwasawa $\lambda$-invariant and the ideal class group of $K$
(Theorems \ref{thm 3} and \ref{main thm of classification lambda=2}).
Such $X_{\widetilde K}$ will be useful for studying the Iwasawa theory of multiple $\Z_p$-extensions.
We remark that our results do not need the assumption that Greenberg's generalized conjecture holds.

\subsection{Notation}\label{notation}

Throughout this paper, we use the following notation.
Let $p$ be an odd prime number, 
$k$ a totally real number field,
$K$ a CM-field such that $K/k$ is a finite abelian extension of degree coprime to $p$,
%$K/k$ a finite abelian extension of degree coprime to $p$ such that $K$ is a CM-field, 
$\Ki/K$ the cyclotomic $\Zp$-extension, and $\tilde{K}$ the maximal multiple $\Zp$-extension over $K$.
For any (finite or infinite) extension $F$ over $\Q$, 
we denote by $L_F$ and $X_F$ the maximal unramified abelian $p$-extension of $F$ and 
the Galois group of $L_F$ over $F$, respectively. 
If $F$ is a finite extension of $\Q$, we denote by $A_F$ the $p$-Sylow subgroup of the ideal class group of $F$.
We identify 
$\Zp[[\Gal(\Ki/K)]]$ with the ring of formal power series $\Zp[[S]]$
by regarding a fixed generator of $\Gal(\Ki/K)$ as $1+S$.
For a character $\chi \colon \Gal(K/k) \to \overline{\Q}_p^\x$, we denote by $\mathcal{O}_\chi$ the ring obtained from $\Zp$ by adjoining all values of $\chi$.
For any $\Gal(K/k)$-module $M$, put
$
M^\chi:=M \ox_{\Zp[\Gal(K/k)]}\mathcal{O}_\chi$.
%{\it i.e.,}
%$M^\chi= \{ x \in M \mid \sigma x=\chi(\sigma) x\}$
We denote by $\mu_p$ the set of all $p$-th roots of unity, and simply by $X/\mathfrak{a}$ a quotient module $X/\mathfrak{a}X$ .
Let $\Lam$ be the ring either
$\mathcal{O}[[S]]$ or $\mathcal{O}[[S,T]]$,
where $\mathcal{O}$ is $\Zp$ or $\mathcal{O}_\chi$.
For any finitely generated torsion $\Lam$-module $X$, 
we call a generator of the characteristic ideal of $X$
a characteristic power series of $X$, and denote it by ${\rm char}_{\Lambda}(X)\in \Lambda$,
which is determined up to $\Lam^\x$.
%we choose a characteristic power series of $X$ and denote it by
%${\rm char}_{\Lam}(X) \in \Lam$.

\subsection{Main theorems of ``{\it Split case}''}

Let $\chi \colon \Gal(K/k) \to \overline{\Q}_p^\x$ be an odd character.
Assume that there is only one prime ideal $\p$ in $k$ above $p$ which satisfies that $\chi(\p)=1$.
Let $\tilde{K}'/K$ be the maximal multiple $\Zp$-extension such that $\Ki \subset \tilde{K}'$ and $\tilde{K}'/\Ki$ is unramified.
We know that $\Gal(\tilde{K}'/K)$ is a $\Gal(K/k)$-module.
Note that if $k=\Q$, then $\tilde{K}=\tilde{K}'$.
By \cite{DDP} (as a consequence of Lemma 1.5 and Proposition 1.6), 
we see that $\Gal(\tilde{K}'/K)^\chi \isom \mathcal{O}_\chi$.
Let $\tilde{K}_\chi$ be the subextension of $\tilde{K}/\Ki$ such that 
%whose Galois group is 
$\Gal(\tilde{K}_\chi/\Ki)$ is isomorphic to $\Gal(\tilde{K}'/K)^\chi$.
We identify $\Zp[[\Gal(\tilde{K}_\chi/K)]]=\Zp[[\Gal(\Ki/K)\x \Gal(\tilde{K}_\chi/\Ki)]]$ with $\Zp[[S,T_1,\ldots,T_{d_\chi}]]$ in a similar way as we did $\Zp[[\Gal(\Ki/K)]]$ with $\Zp[[S]]$, where $d_\chi:=[\mathcal{O}_\chi:\Zp]$.

\begin{thm}\label{thm 1}
%Let $p$ be an odd prime number, $k$ is a totally real field, $K/k$ a finite cyclic abelian extension such that $K$ is a CM-field, and $\chi \colon \Gal(K/k) \to \overline{\Q}_p^\x$ be a faithful character on $\Gal(K/k)$.
Let $p$, $k$, $K$ be as in \S {\rm \ref{notation}}, $\chi \colon \Gal(K/k) \to \overline{\Q}_p^\x$ an odd character.
%\\
%\textcolor{red}{$K/k$ cyclic??, $\chi$ faithful??}
%\\
Assume that 
%$p \nmid [K:k]$ and 
$\mu_p \not\subset K$
and that 
there exists just only one prime ideal $\p$ in $k$ above $p$ which satisfies that $\chi(\p)=1$.
Then $\Gal(K/k)$-module 
$
(X_{\tilde{K}_\chi})_{\Gal(\tilde{K}_\chi/K)}
=
X_{\tilde{K}_\chi}/(S,T_1,\cdots,T_{d_\chi})
$
satisfies
$$
\# \left( X_{\tilde{K}_\chi}/(S,T_1,\cdots,T_{d_\chi}) \right)^\chi
=
\# \mathcal{O}_\chi/f_\chi^*,
$$
where $f_\chi^*$ is the first non-vanishing coefficient of a characteristic power series of the $\mathcal{O}_\chi[[S]]$-module $X_{\Ki}^\chi$.
Moreover, assume that Leopoldt's conjecture holds for the maximal totally real subfield $K^+$ and $p$, then 
there is a canonical isomorphism 
$$
\left((X_{\tilde{K}'})_{\Gal(\tilde{K}'/\Ki)}\right)^\chi
\isom
\left((X_{\tilde{K}_\chi})_{\Gal(\tilde{K}_\chi/\Ki)}\right)^\chi .
%=
%(X_{\tilde{K}_\chi}/(T_1,\cdots,T_{d_\chi}))^\chi.
$$
\end{thm}

\begin{rem}
\begin{rm}
Since Gross's `order of vanishing conjecture' (see Gross \cite[Conjecture 1.15]{Gross82} or Dasgupta, Kakde, and Ventullo \cite[Conjecture 1]{DKV18}, for example) of rank one for $K/k$ holds by the argument in \cite[Proposition 2.13]{Gross82},
$f_\chi^*$ turns out to be the coefficient of degree $1$.
\end{rm}
\end{rem}

\begin{cor}\label{Cor of them 1}
Let $p$ be an odd prime number, $K$ an imaginary abelian finite extension over $\Q$ of degree coprime to $p$ in which $p$ splits completely.
Then, for any odd character $\chi$ of $\Gal(K/k)$, we have
$$
\# \left((X_{\tilde{K}})_{\Gal(\tilde{K}/K)} \right)^\chi
=
\# \mathcal{O}_\chi/f_\chi^*.
$$
%\textcolor{red}
%{
%where $\chi$ runs all odd characters of $\Gal(K/k)$.
%}
%where $f^*$ is the first non-vanishing coefficient of a characteristic power series of $X_{\Ki}^-$.
\end{cor}

In \S \ref{Proof of Thm 2}, we will consider the case where $K$ is an imaginary quadratic field such that $p$ splits completely as $(p)=\P \overline{\P}$ and connect the above corollary with the Galois group $\mathfrak{X}_{\P}(\tilde{K})$ of the maximal abelian $p$-extension of $\tilde{K}$ which is unramified outside all primes above  $\P$.
%Denote by $\mathfrak{X}_{\P}(\tilde{K})$ the Galois group of the maximal abelian $p$-extension of $\tilde{K}$ which unramified outside all primes above $\P$.
It is known that $\mathfrak{X}_{\P}(\tilde{K})$ 
%and $\mathfrak{X}_{\overline{\P}}(\tilde{K})$ are 
is finitely generated and torsion over $\Zp[[S,T]]=\Zp[[\Gal(\tilde{K}/K)]]$ $(T:=T_1)$
by \cite[Theorem 5.3 (ii)]{Rubin91}.
Therefore, we can consider the characteristic ideal 
%$\mathcal{L}_{\P} \subset \Zp[[S,T]]$ 
of $\mathfrak{X}_{\P}(\tilde{K})$ as a $\Zp[[S,T]]$-module, 
which plays an important role in the Iwasawa main conjecture.
We denote by $\lam$ the Iwasawa $\lam$-invariant of $\Ki/K$.
Then it is known that $\mathfrak{X}_{\P}(\tilde{K})$ is generated by $\lam -1$ elements as a $\Zp[[T]]$-module.
Moreover, let $I_{\lam-1}$ be the identity matrix of size $\lam-1$ and $A$ a matrix associated to multiplication by $S$ on $\mathfrak{X}_{\P}(\tilde{K})$ whose entries are in $\Zp[[T]]$.
We will show the following theorem which suggests that the characteristic ideal of $\mathfrak{X}_{\P}(\tilde{K})$ relates to the structure of $X_{\tilde{K}}$.

\begin{thm}\label{thm 2}
Let $p$ be an odd prime number, $K$ an imaginary quadratic field such that $p$ splits completely as $(p)=\P \overline{\P}$ and $\tilde{K}$ the unique $\Zp^2$-extension of $K$.
Assume that 
$L_K \subset \tilde{K}$
or that
the characteristic ideal 
%${\rm char}_{\Zp[[S]]}(X_{\Ki})$
of $X_{\Ki}$ 
does not have any square factor.
Then the characteristic power series
${\rm char}_{\Zp[[S,T]]}(\mathfrak{X}_{\P}(\tilde{K}))$
of $\mathfrak{X}_{\P}(\tilde{K})$
satisfies
\begin{eqnarray*}
&
\left({\rm char}_{\Zp[[S,T]]}(\mathfrak{X}_{\P}(\tilde{K}))\right)\Zp[[S,T]] 
= \left( \det (S \cdot I_{\lam-1}-A)\right)\Zp[[S,T]],
&
\\
&
\left({\rm char}_{\Zp[[S,T]]}(\mathfrak{X}_{\P}(\tilde{K}))\mid_{T=0}\right)\Zp[[S]] 
= \left( \f{{\rm char}_{\Zp[[S]]}(X_{\Ki})}{S}\right)\Zp[[S]]
&
\end{eqnarray*}
as ideals.
In particular, combining these with Corollary \ref{Cor of them 1}, we have
$$
\# \left((X_{\tilde{K}})_{\Gal(\tilde{K}/K)} \right)=
\# \Zp/\left({\rm char}_{\Zp[[S,T]]}(\mathfrak{X}_{\P}(\tilde{K}))\mid_{S=T=0}\right).
$$
\end{thm}

%%%%%%%%%%%%%%%%%%%%%%%%%%%%%%%%%%%%%%%%%%%%%%%%%%%%%%%%%%%%%%%%%%%%%%%%%%%%%%%%%%%%%%%%

\subsection{Main theorems of ``{\it Non-split case}''}
In the latter half of this paper, we consider the case where $K$ is an imaginary quadratic 
field
such that $p$ does not split.
Then it is well known that the number of generators of $X_{\Ki}$ as a $\Zp[[S]]$-module is equal to $\dim_{\Fp} (A_K/p)$.
Our target is the number of generators of $X_{\tilde{K}}$ as a $\Zp[[S,T]]$-module,
which is equal to $\dim_{\Fp} (X_{\tilde{K}}/(p,S,T))$ by Nakayama's lemma.
It is easy to see that
$$
\dim_{\Fp} (A_K/p)-1 \le \dim_{\Fp} (X_{\tilde{K}}/(p,S,T)) \le \dim_{\Fp} (A_K/p)+1
$$
(see \S \ref{a system of generators}).
In \S \ref{a system of generators}, we describe a system of generators of $X_{\tilde{K}}$ and show the following theorem which classifies $\dim_{\Fp} (X_{\tilde{K}}/(p,S,T))$ in the case where $A_K$ is a cyclic abelian group.

\begin{thm}\label{thm 3}
\makeatletter
  \parsep   = 0pt
  \labelsep = .5pt
  \def\@listi{%
     \leftmargin = 20pt \rightmargin = 0pt
     \labelwidth\leftmargin \advance\labelwidth-\labelsep
     \topsep     = 0\baselineskip
     \partopsep  = 0pt \itemsep       = 0pt
     \itemindent = 0pt \listparindent = 10pt}
  \let\@listI\@listi
  \@listi
  \def\@listii{%
     \leftmargin = 20pt \rightmargin = 0pt
     \labelwidth\leftmargin \advance\labelwidth-\labelsep
     \topsep     = 0pt \partopsep     = 0pt \itemsep   = 0pt
     \itemindent = 0pt \listparindent = 10pt}
  \let\@listiii\@listii
  \let\@listiv\@listii
  \let\@listv\@listii
  \let\@listvi\@listii
  \makeatother
Let $p$ be an odd prime number and $K$ an imaginary quadratic field such that $p$ does not split.
\begin{itemize}
\item[{\rm (i)}]
{\rm (trivial case)}
%Suppose that $A_K=0$, then $X_{\tilde{K}}=0$.
%On the other hand, 
Assume that $L_K \cap \tilde{K}=K$, then 
$
\dim_{\Fp} (X_{\tilde{K}}/(p,S,T))=\dim_{\Fp} (A_K/p).
$
\item[{\rm (ii)}]
Suppose that $L_K \cap \tilde{K} \neq K$, and that $\dim_{\Fp} (A_K/p)=1$.
\begin{itemize}
\item[{\rm (ii-a)}]
If $\lam=1$, then 
$
\dim_{\Fp} (X_{\tilde{K}}/(p,S,T))=1.
$
\item[{\rm (ii-b)}]
If $\lam \ge 2$, then
$$
\dim_{\Fp} (X_{\tilde{K}}/(p,S,T))=
\begin{cases}
1\ \ \text{if $L_K \subset \tilde{K}$},
\\
2\ \ \text{otherwise}.
\end{cases}
$$
\end{itemize}
\end{itemize}
\end{thm}

From the point of view of applications, 
we are interested in the condition for $X_{\tilde{K}}$ to be $\Zp[[S,T]]$-cyclic.
In \S \ref{Classification in the case lam=2}, we classify when $X_{\tilde{K}}$ is $\Zp[[S,T]]$-cyclic in a certain case with $\dim_{\Fp} (A_K/p)=1$ and $\lam=2$
(Theorem \ref{main thm of classification lambda=2}).
%All computations in \S \ref{Classification in the case lam=2} are not complicated.
We are going to give some numerical examples about it and introduce a method of calculating these examples
in the forth coming paper.

%%%%%%%%%%%%%%%%%%%%%%%%%%%%%%%%%%%%%%%%%%%%%%%%%%%%%%%%%%
%%%%%%%%%%%%%%%%%%%%%%%%%%%%%%%%%%%%%%%%%%%%%%%%%%%%%%%%%%
%%%%%%%%%%%%%%%%%%%%%%%%%%%%%%%%%%%%%%%%%%%%%%%%%%%%%%%%%%

%%%%%%%%%%%%%%%%%%%%%%%%%%%%%%%%%%%%%%%%%%%%%%%%%%%%%%%%%%
%%%%%%%%%%%%%%%%%%%%%%%%%%%%%%%%%%%%%%%%%%%%%%%%%%%%%%%%%%
%%%%%%%%%%%%%%%%%%%%%%%%%%%%%%%%%%%%%%%%%%%%%%%%%%%%%%%%%%

\section{Proof of Theorem \ref{thm 1}}\label{Proof of Theorem 1}
\subsection{Preliminary from general module theory}\label{Preliminary}

For a module $M$ and a morphism $\varphi \colon M \to M$, we define $M[\varphi]:=\Ker(\varphi)$ and $\varphi M:={\rm Im}(M)$.

\begin{prop}\label{ker-im 1}
Let 
\begin{eqnarray}\label{CD of ker-im 1}
\begin{CD}
0 @>>> L @>>> M @>>> N @>>> 0
\\
& & @V \alp VV  @V \beta VV   @V \gam VV & 
\\
0 @>>> L @>>> M @>>> N @>>> 0
\end{CD}
\end{eqnarray}
be an exact commutative diagram of modules.
Regard $L$ as a submodule in $M$.
If $\gam N=0$ and $\# (M/M[\beta]+\beta M) <\infty$, then the following equation
\begin{eqnarray}\label{EQ of ker-im 1}
\# \f{M}{M[\beta]+\beta M}
=
\# \f{L}{L[\alp]+\alp L}
\cdot
\# \f{M[\beta] \cap \beta M}{L[\alp] \cap \alp L}
\end{eqnarray}
holds.
\end{prop}
 
\begin{proof}
We consider an exact commutative diagram
$$
\begin{CD}
0 @>>> L[\alp] \op \alp L @>>> M[\beta] \op \beta M @>>> \f{M[\beta]}{L[\alp]} \op \f{\beta M}{\alp L} @>>> 0
\\
& & @V l VV  @V m VV @V n VV  
\\
0 @>>> L @>>> M @>>> N @>>> 0,
\end{CD}
$$
where the vertical maps are defined by the difference between the two components, for example, 
$l \colon (x,y) \mapsto x-y$.
We will show $\Ker(n) \isom \Coker(n)$.
Since $\gam N=0$, the image of $\beta M/\alp L$ by $n$ is zero.
Also the natural map $M[\beta]/L[\alp] \to N[\gam]=N$ is injective by the snake lemma.
Hence we obtain
$\Ker(n)=\beta M/\alp L.$
On the other hand,
%since $\beta M \subset L$, which followed from $\gam N=0$, we obtain
we see
$$
\Coker(n)
=
M/(L+M[\beta]+\beta M)
=
M/(L+M[\beta]).
$$
Hence an exact sequence 
$0 \to L+M[\beta] \to M \up{\beta}{\to} \beta M/\alp L \to 0$
induces an isomorphism 
$\beta M/\alp L \isom 
%M/(L+\beta M)=
\Coker(n)$.
Therefore we obtain $\Ker(n) \isom \Coker(n)$.
Applying the snake lemma to the first exact commutative diagram, we have an exact sequence
\begin{eqnarray*}\label{EX of ker-im 1}
0 \to \f{M[\beta] \cap \beta M}{L[\alp] \cap \alp L} \to \Ker(n) \to \f{L}{L[\alp]+\alp L} \to \f{M}{M[\beta]+\beta M} \to \Coker(n) \to 0.
\end{eqnarray*}
From this and the assumption that $\# M/M[\beta]+\beta M <\infty$, we have the claim.
\end{proof}

%%%%%%%%%%%%%%%%%%%%%%%%%%%%%%%%%%%%%%%%%%%%%%%%%%%%%%%%%%%%%%%%%%%%%%%%%%%%%%

%%%%%%%%%%%%%%%%%%%%%%%%%%%%%%%%%%%%%%%%%%%%%%%%%%%%%%%%%%%%%%%%%%%%%%%%%%%%%%

\begin{cor}\label{ker-im 2}
For the exact commutative diagram (\ref{CD of ker-im 1}) in Proposition \ref{ker-im 1}, 
suppose that $\# (M/M[\beta]+\beta M) <\infty$.
\makeatletter
  \parsep   = 0pt
  \labelsep = .5pt
  \def\@listi{%
     \leftmargin = 20pt \rightmargin = 0pt
     \labelwidth\leftmargin \advance\labelwidth-\labelsep
     \topsep     = 0\baselineskip
     \partopsep  = 0pt \itemsep       = 0pt
     \itemindent = 0pt \listparindent = 10pt}
  \let\@listI\@listi
  \@listi
  \def\@listii{%
     \leftmargin = 20pt \rightmargin = 0pt
     \labelwidth\leftmargin \advance\labelwidth-\labelsep
     \topsep     = 0pt \partopsep     = 0pt \itemsep   = 0pt
     \itemindent = 0pt \listparindent = 10pt}
  \let\@listiii\@listii
  \let\@listiv\@listii
  \let\@listv\@listii
  \let\@listvi\@listii
  \makeatother
\ 
\begin{itemize}
\item[{\rm (i)}]
Assume that there exists some integer $n>0$ such that $\gam^n N=0$.
Then the equation {\rm (\ref{EQ of ker-im 1})} holds.
In particular, if all modules in {\rm (\ref{CD of ker-im 1})} are $\Zp[[S]]$-modules, the maps $\alp$, $\beta$, $\gam$ are multiplication by $S$, and $N$ is finite, then the equation {\rm (\ref{EQ of ker-im 1})} holds.
\item[{\rm (ii)}]
Assume that $\gam N=0$, $M[\beta]+\beta M/\beta M$ is torsion-free, and $L/\alp L$ is finite.
Then
$$
\# \f{M}{M[\beta]+\beta M}=\# \f{L}{\alp L}.
$$
\end{itemize}
\end{cor}

\begin{proof}
(i) 
Put $M_0:=M$,
$M_i:=\beta^i M+L$ ($i \ge 0$).
Note that $M_n=L$, since
$0=\gam^n N \isom \beta^n M+L/L$.
Define $N_i$ by the exact sequence
$$
0 \to M_{i+1} \to M_i \to N_i \to 0
\ \ 
(i \ge 0).
$$
Then we have $\beta N_i=0$.
Therefore we can easily check that
$\# (M_i/M_i[\beta]+\beta M_i) <\infty$ for $i=0,\ldots,n$, 
using the same method as in the proof of Proposition \ref{ker-im 1} and the induction on $i$.
Therefore we can apply Proposition \ref{ker-im 1} to the above exact sequence, so that
$$
\# \f{M_i}{M_i[\beta]+\beta M_i}
=
\# \f{M_{i+1}}{M_{i+1}[\beta]+\beta M_{i+1}}
\cdot
\# \f{M_i[\beta] \cap \beta M_i}{M_{i+1}[\beta] \cap \beta M_{i+1}}.
$$
Taking the products from $i=0$ to $i=n-1$, we have the claim since $M_n=L$.
\\
(ii)
Note that $\beta M \subset L$ and $L[\alp]/L[\alp] \cap \alp L$ is finite by the assumption.
Hence the exact sequence 
$$
0 \to \f{M[\beta] \cap \beta M}{L[\alp] \cap \alp L} \to \f{L[\alp]}{L[\alp] \cap \alp L} \to 
\f{M[\beta]}{M[\beta] \cap \beta M}
$$
implies that
$({M[\beta] \cap \beta M})/({L[\alp] \cap \alp L}) \isom {L[\alp]+\alp L}/{\alp L}$, 
since ${M[\beta]+\beta M}/{\beta M}$ is torsion-free.
Therefore, 
$$
\# \f{M}{M[\beta]+\beta M}
=
\# \f{L}{L[\alp]+\alp L}
\cdot
\# \f{M[\beta] \cap \beta M}{L[\alp] \cap \alp L}
=
\# \f{L}{\alp L}
$$
by (i).
This completes the proof.
\end{proof}

%%%%%%%%%%%%%%%%%%%%%%%%%%%%%%%%%%%%%%%%%%%%%%%%%%%%%%%%%%%%%%%%%%%%%%%%%%%%%%

%%%%%%%%%%%%%%%%%%%%%%%%%%%%%%%%%%%%%%%%%%%%%%%%%%%%%%%%%%%%%%%%%%%%%%%%%%%%%%

Let $\mathcal{O}$ be the ring either $\Zp$ or $\mathcal{O}_\chi$.
Applying the corollary to $\mathcal{O}[[S]]$-modules, we obtain the following corollary.
Let $X$ be a finitely generated torsion 
$\mathcal{O}[[S]]$-module which has no nontrivial finite $\mathcal{O}[[S]]$-submodules.
Then, by the structure theorem of finitely generated torsion 
$\mathcal{O}[[S]]$-modules 
(see \cite[Theorem 13.12]{Was}),
there is an exact sequence of 
$\mathcal{O}[[S]]$-modules
\begin{eqnarray}\label{str thm of X}
0 \to X \to E \to C \to 0
\end{eqnarray}
such that $E$ is described as $E=\bigop_j \mathcal{O}[[S]]/(f_j(S)^{n_j})$ and $C$ is finite.
Here, each $f_j(S) \in \mathcal{O}[S]$ is an irreducible distinguished polynomial or a uniformizer in $\mathcal{O}$.
%the prime number $p$.

\begin{cor}\label{X/X[S]+SX and lead. coeff.}
With the notation as above, we denote the first non-vanishing coefficient of $\prod_j f_j(S)^{n_j}$ by $f^*$.
If $S^2$ does not divide $f_j(S)^{n_j}$ for any $j$, then
$$
\#(X/X[S]+SX)
=
\# \mathcal{O}/f^*.
$$
\end{cor}

\begin{proof}
In the case where $E=\mathcal{O}[[S]]/(f(S)^{n})$ with $S \nmid f(S)$, we have
$$
E[S]=0,\ \ E/E[S]+SE 
\isom 
\mathcal{O}/(f(0)^{n}).
$$
And also, in the case where $E=\mathcal{O}[[S]]/(S)$, we see
$$
SE=0,\ \ E/E[S]+SE =0.
$$
Hence, in both cases, $E[S] \cap SE=0$.
Applying Corollary \ref{ker-im 2} (i) to the exact sequence (\ref{str thm of X}), we have
$
\#(X/X[S]+SX)=\#(E/E[S]+SE)
=
\# \mathcal{O}/f^*.
$
This completes the proof.
\end{proof}

%%%%%%%%%%%%%%%%%%%%%%%%%%%%%%%%%%%%%%%%%%%%%%%%%%%%%%%%%%%%%%%%%%%%%%%%%%%%%%
%%%%%%%%%%%%%%%%%%%%%%%%%%%%%%%%%%%%%%%%%%%%%%%%%%%%%%%%%%%%%%%%%%%%%%%%%%%%%%

%%%%%%%%%%%%%%%%%%%%%%%%%%%%%%%%%%%%%%%%%%%%%%%%%%%%%%%%%%%%%%%%%%%%%%%%%%%%%%
%%%%%%%%%%%%%%%%%%%%%%%%%%%%%%%%%%%%%%%%%%%%%%%%%%%%%%%%%%%%%%%%%%%%%%%%%%%%%%
%%%%%%%%%%%%%%%%%%%%%%%%%%%%%%%%%%%%%%%%%%%%%%%%%%%%%%%%%%%%%%%%%%%%%%%%%%%%%%
%%%%%%%%%%%%%%%%%%%%%%%%%%%%%%%%%%%%%%%%%%%%%%%%%%%%%%%%%%%%%%%%%%%%%%%%%%%%%%
%%%%%%%%%%%%%%%%%%%%%%%%%%%%%%%%%%%%%%%%%%%%%%%%%%%%%%%%%%%%%%%%%%%%%%%%%%%%%%

%%%%%%%%%%%%%%%%%%%%%%%%%%%%%%%%%%%%%%%%%%%%%%%%%%%%%%%%%%%%%%%%%%%%%%%%%%%%%%
%%%%%%%%%%%%%%%%%%%%%%%%%%%%%%%%%%%%%%%%%%%%%%%%%%%%%%%%%%%%%%%%%%%%%%%%%%%%%%
%%%%%%%%%%%%%%%%%%%%%%%%%%%%%%%%%%%%%%%%%%%%%%%%%%%%%%%%%%%%%%%%%%%%%%%%%%%%%%
%%%%%%%%%%%%%%%%%%%%%%%%%%%%%%%%%%%%%%%%%%%%%%%%%%%%%%%%%%%%%%%%%%%%%%%%%%%%%%
%%%%%%%%%%%%%%%%%%%%%%%%%%%%%%%%%%%%%%%%%%%%%%%%%%%%%%%%%%%%%%%%%%%%%%%%%%%%%%

\subsection{Galois coinvariants}\label{Galois coinvariants}

We use the notation in \S \ref{notation}.
Let $\chi \colon \Gal(K/k) \to \overline{\Q}_p^\x$ be an odd character.
Suppose that the assumption in Theorem \ref{thm 1}.
In other words, 
we assume that $\mu_p \not\subset K$ and 
there is only one prime ideal $\p$ lying above $p$ in $k$ which satisfies $\chi(\p)=1$.
Recall that $\tilde{K}'$ is the maximal multiple $\Zp$-extension such that $\Ki \subset \tilde{K}'$ and $\tilde{K}'/\Ki$ is unramified.
Then there is a unique abelian extension $\tilde{K}_\chi/K$ contained in $\tilde{K}'$ such that
$$
\Gal(\tilde{K}_\chi/\Ki) \isom \mathcal{O}_\chi
$$
as $\Gal(K/k)$-modules by 
\cite[Lemma 1.5]{DDP}.
Also, recall that we identify 
$\Zp[[\Gal(\tilde{K}_\chi/K)]]$ with $\Zp[[S,T_1,\ldots,T_{d_\chi}]]$,
%$
%\Zp[[\Gal(\tilde{K}_\chi/\Ki)]]=\Zp[[T_1,\ldots,T_{d_\chi}]]
%$,
where $d_\chi:=[\mathcal{O}_\chi:\Zp]$.
%Note that the maximal abelian subextension in $L_{\Ki}/K$ is the field fixed by $SX_{\Ki}$.
Since $L_{\tilde{K}_\chi}/k$ is a Galois extension because of the maximality of $L_{\tilde{K}_\chi}$,
$\Gal(\tilde{K}_\chi/k)$ acts on $X_{\tilde{K}_\chi}$ by the inner product:
$\alp(x):=\tilde{\alp}x\tilde{\alp}^{-1}$ for $\alp \in \Gal(\tilde{K}_\chi/k)$ and $x \in X_{\tilde{K}_\chi}$, where $\tilde{\alp} \in \Gal(L_{\tilde{K}_\chi}/k)$ is a fixed extension of $\alp$.
%Also, we can define an action of $\Gal(\Ki/k)$ on $\Gal(L_{\tilde{K}_\chi}/\Ki)$ in the same way.
%Note that the first action is independent of the choice of extensions $\tilde{\alp}$, but not the second is, since $\Gal(L_{\tilde{K}_\chi}/\Ki)$ is not abelian.
%
%(間違い!)
%By this action, the subgroup $X_{\tilde{K}_\chi}$ of $\Gal(L_{\tilde{K}_\chi}/\Ki)$ becomes a $\Gal(\Ki/k)$-module.
%
Note that the action is independent of the choice of extensions $\tilde{\alp}$.
By this action, $X_{\tilde{K}_\chi}$ becomes a $\Gal(\tilde{K}_\chi/k)$-module.
We have two exact sequences
\begin{eqnarray}\label{gemstone tilde{K}_chi}
1 \to X_{\tilde{K}_\chi} \to \Gal(L_{\tilde{K}_\chi}/\Ki) \to \mathcal{O}_\chi \to 1
\end{eqnarray}
and
\begin{eqnarray}\label{gemstone tilde{K}'}
1 \to X_{\tilde{K}'} \to \Gal(L_{\tilde{K}'}/\Ki) \to \Gal(\tilde{K}'/\Ki) \to 1.
\end{eqnarray}
Put $Y:=X_{\tilde{K}_\chi}/(T_1,\ldots,T_{d_\chi})$.
The first part of the following lemma is a partial generalization of Ozaki \cite[Lemma 1]{Ozaki01}.
% in the case 
%where
%$K$ is an imaginary quadratic field, but we prove in more general situations.

\begin{lem}\label{Ozaki01 ex seq (sp case)}
\makeatletter
  \parsep   = 0pt
  \labelsep = .5pt
  \def\@listi{%
     \leftmargin = 20pt \rightmargin = 0pt
     \labelwidth\leftmargin \advance\labelwidth-\labelsep
     \topsep     = 0\baselineskip
     \partopsep  = 0pt \itemsep       = 0pt
     \itemindent = 0pt \listparindent = 10pt}
  \let\@listI\@listi
  \@listi
  \def\@listii{%
     \leftmargin = 20pt \rightmargin = 0pt
     \labelwidth\leftmargin \advance\labelwidth-\labelsep
     \topsep     = 0pt \partopsep     = 0pt \itemsep   = 0pt
     \itemindent = 0pt \listparindent = 10pt}
  \let\@listiii\@listii
  \let\@listiv\@listii
  \let\@listv\@listii
  \let\@listvi\@listii
  \makeatother
\ 
\begin{itemize}
\item[{\rm (i)}]
There is an exact sequence of $\Zp [[S]]$-modules
\begin{eqnarray}\label{ex chi}
0 \to Y^{\chi} \to X_{\Ki}^{\chi} \to \mathcal{O}_\chi \to 0.
\end{eqnarray}
\item[{\rm (ii)}]
Assume that Leopoldt's conjecture holds for $K^+$ and $p$, then the latter half of Theorem \ref{thm 1} holds, in other words,
there is a canonical isomorphism
$\left((X_{\tilde{K}'})_{\Gal(\tilde{K}'/\Ki)}\right)^\chi \isom Y^\chi$.
\end{itemize}
\end{lem}

\begin{proof}
(i) 
Taking the Hochschild-Serre spectral sequence of (\ref{gemstone tilde{K}_chi}), we obtain an exact sequence
 of $\Gal(\Ki/k)$-modules
% and $\Gal(\Ki/k)$-homomorphisms
$$
H_2(\mathcal{O}_\chi, \Zp) \to (X_{\tilde{K}_\chi})_{\Gal(\tilde{K}_\chi/\Ki)} \to X_{\Ki} \to \mathcal{O}_\chi \to 0.
$$
Since $\Gal(K/K^+)$ acts trivially on
$H_2(\mathcal{O}_\chi, \Zp) \isom \mathcal{O}_\chi \wedge_{\Zp} \mathcal{O}_\chi$ 
(here, $\Gal(K/K^+)$ acts on the right hand side diagonally), its $\chi$-component is trivial.
Hence we obtain the claim.
\\
(ii) 
In the same way,
we have
$$
H_2(\Gal(\tilde{K}'/\Ki),\Zp) \to (X_{\tilde{K}'})_{\Gal(\tilde{K}'/\Ki)} \to X_{\Ki} \to \Gal(\tilde{K}'/\Ki) \to 0
$$
from (\ref{gemstone tilde{K}'}).
Combining this with (i), we have only to show that
$\left (\Gal(\tilde{K}'/\Ki) \wedge_{\Zp} \Gal(\tilde{K}'/\Ki)\right)^-=0$, since $\Gal(\tilde{K}'/\Ki)^\chi=\Gal(\tilde{K}_\chi/\Ki)$.
For this, it is enough to show that $\Gal(\tilde{K}'/\Ki)^+=0$.
Assume that it does not hold.
Then $\Gal(\tilde{K}'/\Ki)^+$ is nontrivial and torsion-free, which implies that there exists a $\Zp$-extension of $K^+$ different from the cyclotomic $\Zp$-extension.
This contradicts our assumption that Leopoldt's conjecture holds.
\end{proof}

It is known
%
%
%%%\footnote
%%%{\footnotesize{
%%%We may add more details in latter varsions or the submitted paper for convenience.
%%%}}
%
%
that 
$(X_{\Ki}/X_{\Ki}[S]+SX_{\Ki})^\chi$ is finite and 
$(X_{\Ki}[S]+SX_{\Ki}/SX_{\Ki})^\chi$ is torsion-free
(see Kurihara \cite{Kurihara11}, for example),
since Gross's order of vanishing conjecture holds.
% (the rank one) Gross' conjecture for $K/k$ is proved by Ventullo \cite{Ventullo15}.
%\textcolor{red}{
%Next, we show that $(X_{\Ki}/X_{\Ki}[S]+SX_{\Ki})^\chi$ is finite and that $(X_{\Ki}[S]+SX_{\Ki}/SX_{\Ki})^\chi$ is torsion-free.
%Fix a prime ideal $\P$ in $K$ lying above $\p$.
%Let $D_{\P}(L'/\Ki)$ be a decomposition subgroup in $\Gal(L'/\Ki)$ of a prime lying above $\P$
%(it is not depends on the choice of a prime lying above $\P$), and put
%$D^\chi:=\left(\bigop_{\P \mid \p}D_{\P}(L'/\Ki)\right)^\chi$.
%Then $D^\chi $ is regarded as a submodule in $\Gal(L'/\Ki)^\chi$ of finite index, since $\mu_p \not\subset K$
%(see \cite{Kurihara11}).
%Hence, it is a free $\mathcal{O}_\chi$-module of rank $1$,
%On the other hand, define $c$ by the maximal number such that $\p$ splits completely in $c$-th layer $k_c$ in the cyclotomic $\Zp$-extension $\ki/k$.
%Then, the image of the canonical map 
%$X_{\Ki}[S]^\chi \to X_{\Ki}^\chi/SX_{\Ki}^\chi$ coincides with $p^c D^\chi$,
%again since $\mu_p \not\subset K$
%(see \cite{Kurihara11}).
%This implies that 
%$(X_{\Ki}/X_{\Ki}[S]+SX_{\Ki})^\chi$ 
%is finite and
%$(X_{\Ki}[S]+SX_{\Ki}/SX_{\Ki})^\chi$ 
%is torsion-free.
%}
Applying Corollary \ref{ker-im 2} (ii) to the exact sequence (\ref{ex chi}), we have
\begin{eqnarray*}
\# (X_{\tilde{K}_\chi}/(S,T_1,\ldots,T_d))^\chi
&=&
\# (Y/SY)^\chi
%\\
%&=&
=
\# Y^\chi/SY^\chi
\\
&=&
\# (X_{\Ki}/X_{\Ki}[S]+SX_{\Ki})^\chi,
\end{eqnarray*}
since the actions of $\Gal(K/k)$ and $\Gal(\Ki/K)$ are commutative.
Again, since the Gross's conjecture holds,
%$K$ satisfies Gross' conjecture, 
we can apply Corollary \ref{X/X[S]+SX and lead. coeff.} to get
$$
\# (X_{\tilde{K}_\chi}/(S,T_1,\ldots,T_d))^\chi
=
\# \mathcal{O}_\chi/f_\chi^*.
$$
This completes the proof of Theorem \ref{thm 1}.

\begin{rem}
\begin{rm}
Under several assumptions, we can calculate the value $f^{*}$ in Corollary \ref{Cor of them 1}.
More precisely,
let $K$ be an imaginary quadratic field and $p$ an odd prime number.
We suppose that $L_K \subset \widetilde{K}$.
We also suppose that $p\geq 5$ if $p$ does not split in $K$. 
Then we have
$$
[\mathrm{Gal}(\widetilde{K}/K) : \mathfrak{D}]=\# (\mathbb{Z}_p/f^{*}),
$$
where 
%$f^{*}$ is the first non-vanishing coefficient of a characteristic power series of $X_{K_{\infty}}$ and
$\mathfrak{D}$ is the decomposition group in $\mathrm{Gal}(\widetilde{K}/K)$ of a prime lying above $p$ (Murakami \cite[Proposition 3.4]{Murakami}).
\end{rm}
\end{rem}

%%%%%%%%%%%%%%%%%%%%%%%%%%%%%%%%%%%%%%%%%%%%%%%%%%%%%%%%%%%%%%%%%%%%%%%%%%%%%%
%%%%%%%%%%%%%%%%%%%%%%%%%%%%%%%%%%%%%%%%%%%%%%%%%%%%%%%%%%%%%%%%%%%%%%%%%%%%%%
%%%%%%%%%%%%%%%%%%%%%%%%%%%%%%%%%%%%%%%%%%%%%%%%%%%%%%%%%%%%%%%%%%%%%%%%%%%%%%

%%%%%%%%%%%%%%%%%%%%%%%%%%%%%%%%%%%%%%%%%%%%%%%%%%%%%%%%%%%%%%%%%%%%%%%%%%%%%%
%%%%%%%%%%%%%%%%%%%%%%%%%%%%%%%%%%%%%%%%%%%%%%%%%%%%%%%%%%%%%%%%%%%%%%%%%%%%%%
%%%%%%%%%%%%%%%%%%%%%%%%%%%%%%%%%%%%%%%%%%%%%%%%%%%%%%%%%%%%%%%%%%%%%%%%%%%%%%

\section{Proof of Theorem \ref{thm 2}}\label{Proof of Thm 2}

We consider the case where $K$ is an imaginary quadratic field such that $p$ splits completely as $(p)=\P \overline{\P}$.
Then $\tilde{K}'=\tilde{K}$.
Let $\Lam$ be the ring either $\Zp[[S]]$ or $\Zp[[S,T]]$.
Recall that, for any finitely generated torsion $\Lam$-module $X$, we chose a characteristic power series ${\rm char}_{\Lam}(X) \in \Lam$ of $X$.
A $\Lam$-module $X$ is called pseudo-null if $X$ has two relatively prime annihilators in $\Lam$.

\begin {lem}\label{ann and char}
\makeatletter
  \parsep   = 0pt
  \labelsep = .5pt
  \def\@listi{%
     \leftmargin = 20pt \rightmargin = 0pt
     \labelwidth\leftmargin \advance\labelwidth-\labelsep
     \topsep     = 0\baselineskip
     \partopsep  = 0pt \itemsep       = 0pt
     \itemindent = 0pt \listparindent = 10pt}
  \let\@listI\@listi
  \@listi
  \def\@listii{%
     \leftmargin = 20pt \rightmargin = 0pt
     \labelwidth\leftmargin \advance\labelwidth-\labelsep
     \topsep     = 0pt \partopsep     = 0pt \itemsep   = 0pt
     \itemindent = 0pt \listparindent = 10pt}
  \let\@listiii\@listii
  \let\@listiv\@listii
  \let\@listv\@listii
  \let\@listvi\@listii
  \makeatother
Let 
$0 \to X \to E \to C \to 0$
be an exact sequence of finitely generated torsion $\Lam$-modules.
For a $\Lam$-module $M$, denote by ${\rm Ann}_{\Lam}(M)$ the annihilator ideal of $M$.
\begin{itemize}
\item[{\rm (i)}] 
If $E$ has no non-trivial pseudo-null $\Lam$-submodules and $C$ is a pseudo-null $\Lam$-module, then we have 
${\rm Ann}_{\Lam}(X) = {\rm Ann}_{\Lam}(E)$.
\item[{\rm (ii)}] 
Furthermore, suppose that 
$E =\bigop^s_{i=1} \Lam/\p_i^{n_i}$
where each $\p_i$ is a height one prime ideal with $\p_i \neq \p_j$ for $i \neq j$.
Then we have 
${\rm Ann}_{\Lam}(X)= ({\rm char}_{\Lam}(X))$.
\end{itemize}
%Let $X$ be a finitely generated torsion $\Lam$-module such that $X$ has no nontrivial pseudo-null submodules.
%Suppose that there is an exact sequence
%\begin{eqnarray}\label{str seq}
%0 \to X \to \bigop^s_{j=1} \Lam/\p_j^{n_j} \to C \to 0
%\end{eqnarray}
%such that $C$ is a pseudo-null $\Lam$-module and each $\p_j$ is a prime ideal of height 1.
%Assume that $\p_j \neq \p_k$ for any $j \neq k$.
%Then the annihilator ideal
%of $X$ is contained in the characteristic ideal ${\rm char}_{\Lam}(X) \Lam$.
\end{lem}

\begin{proof}
(i) 
The inclusion 
${\rm Ann}_{\Lam}(X) \supset {\rm Ann}_{\Lam}(E)$
is obvious. 
We will show the other inclusion.
Let $f \in \Lam$.
Then we have an exact sequence
$$
0 \to X[f] \to E[f] \to C[f] \to X/f \to E/f \to C/f \to 0.
$$
Suppose that $f \in {\rm Ann}_{\Lam}(X)$.
Then we have $X/f=X$, hence we have an commutative diagram
$$
\begin{CD}
0 @>>> X @>>> E @>>> C @>>> 0
\\
 & & @|  @VVV  @VVV  
\\
0 @>>> X @>>> E/f @>>> C/f @>>> 0.
\end{CD}
$$
By snake lemma, we have 
$0 \to fE \to fC \to 0$, hence $fE$ is a pseudo-null $\Lam$-module.
Since $E$ has no non-trivial pseudo-null $\Lam$-submodules, we have $fE = 0$.
Thus we have  $f \in {\rm Ann}_{\Lam}(E)$, hence we have proved ${\rm Ann}_{\Lam}(X) \subset {\rm Ann}_{\Lam}(E)$.
\\
(ii) 
This follows easily since ${\rm Ann}_{\Lam}(E) = ({\rm char}_{\Lam}(X))$ in this case.
%Let 
%$f$ be in the annihilator ideal of $X$.
%Applying the snake lemma to the exact sequence (\ref{str seq}), we obtain the exactness of 
%$$
%C[f] \to X \to \bigop_j \Lam/(f, \p_j^{n_j}) \to C/f \to 0.
%$$
%Both $C[f]$ and $C/f$ are pseudo-null.
%Therefore, $f \in \p_j^{n_j}$ by the uniqueness of the structure theorem
%(see \cite[Chapter VII \S 4.4 Theorem 5]{Bour comm alg}),
%so that $f \in \prod_j \p_j^{n_j}$ by the assumption that $\p_j \neq \p_k$ ($j \neq k$).
\end{proof}

Let $\mathfrak{X}_{\P}(\tilde{K})$ be the Galois group of the maximal $p$-extension of $\tilde{K}$ unramified outside all primes above $\P$.
Put
$$
\mathcal{F}_{\P}(S,T):={\rm char}_{\Zp[[S,T]]}(\mathfrak{X}_{\P}(\tilde{K}))
\ \ 
{\rm and}
\ \ 
F(S):={\rm char}_{\Zp[[S]]}(X_{\Ki}) \in \Zp[[S]]
$$
for simplicity.
We may assume that $F(S)$ is a distinguished polynomial with degree $\lam$, where $\lam$ is the Iwasawa $\lam$-invariant of $\Ki/K$.
From Lemma \ref{Ozaki01 ex seq (sp case)} (i), we have the exact sequence 
\begin{eqnarray}\label{Ozaki01 ex seq (imag. quad. sp case)}
0 \to X_{\tilde{K}}/T \to X_{\Ki} \to \Zp \to 0.
\end{eqnarray}
Therefore $F(S)$ is written as $F(S)=SF^*(S)$ for some distinguished polynomial $F^*(S)$ which is coprime to $S$ with degree $\lam-1$, 
and $X_{\tilde{K}}/T$ has no nontrivial finite $\Zp[[S]]$-submodules.
On the other hand, by Fujii \cite[Lemma 3]{Fujii14}, 
$\mathfrak{X}_{\P}(\tilde{K})$ has no nontrivial pseudo-null submodules and
\begin{eqnarray}\label{mathfrak{X}/T isom tilde{X}/T}
\mathfrak{X}_{\P}(\tilde{K})/T \isom X_{\tilde{K}}/T,
% \isom \Zp^{\lam-1}
\end{eqnarray}
since $K$ is an imaginary quadratic field.
Therefore $\mathfrak{X}_{\P}(\tilde{K})/T$ is $\Zp[[S]]$-torsion.
% and has no nontrivial finite $\Zp[[S]]$-submodules.
This yields that $\mathfrak{X}_{\P}(\tilde{K})[T]$ is a pseudo-null $\Zp[[S,T]]$-module by 
Perrin-Riou \cite[Chapitre I Lemme 4.2 in page 12]{Perrin-Riou84},
so that $\mathfrak{X}_{\P}(\tilde{K})[T]=0$.
Hence we have
$$
\left(\mathcal{F}_{\P}(S,0)\right)
=
\left(\f{{\rm char}_{\Zp[[S]]}(\mathfrak{X}_{\P}(\tilde{K})/T)}{{\rm char}_{\Zp[[S]]}(\mathfrak{X}_{\P}(\tilde{K})[T])}\right)
=
\left(F^*(S)\right)
$$
as ideals in $\Zp[[S]]$, again by the same lemma in \cite{Perrin-Riou84}.
%By the same way as in 
%Murakami \cite[Lemma 3.3]{Murakami} and 
By (\ref{mathfrak{X}/T isom tilde{X}/T}), 
we can take $\lam-1$ generators 
$\xi_1,\ldots, \xi_{\lam-1}$
of $\mathfrak{X}_{\P}(\tilde{K})$ as a $\Zp[[T]]$-module.
Define a matrix $A$ and $f(S,T) \in \Zp[[T]][S]$ by
\begin{eqnarray*}
S
\begin{bmatrix}
\xi_1
\\
\vdots
\\
\xi_{\lam-1}
\end{bmatrix}
=
A
\begin{bmatrix}
\xi_1
\\
\vdots
\\
\xi_{\lam-1}
\end{bmatrix}
\end{eqnarray*}
and $f(S,T)=\det(S\cdot I_{\lam-1}-A)$, respectively.
Then $f(S,T)$ annihilates $\mathfrak{X}_{\P}(\tilde{K})$, and also $X_{\tilde{K}}$.
Hence, $f(S,0)$ annihilates $X_{\tilde{K}}/T$.
Note that the degree of $f(S,T)$ with respect to $S$ is $\lam-1$.
Now we show Theorem \ref{thm 2}.

\begin{thm}{\rm (Theorem \ref{thm 2})}
Suppose that 
$L_K \subset \tilde{K}$
or $F(S)$ does not have any square factor.
Then the following equalities
$$
\mathcal{F}_{\P}(S,T)\Zp[[S,T]]=f(S,T)\Zp[[S,T]]
\ \ \text{and} \ \ 
\mathcal{F}_{\P}(S,0)\Zp[[S]]=F^*(S)\Zp[[S]]
%\left( {\rm char}_{\Zp[[S,T]]}(\mathfrak{X}_{\P}(\tilde{K})) \right)
%=
%\left( f(S,T) \right)
$$
as ideals 
hold.
\end{thm}

\begin{proof}
First, suppose that $F(S)$ does not have any square root.
Combining the equation $(F^*(S))=(\mathcal{F}_{\P}(S,0))$ with this assumption, 
we see that $\mathcal{F}_{\P}(S,T)$ does not have any square factor.
By Lemma \ref{ann and char} (ii), 
$$
f(S,T)=\mathcal{F}_{\P}(S,T)g(S,T)
\ \ 
{\rm and}
\ \ 
F^*(S) \mid f(S,0)
$$
for some $g(S,T) \in \Zp[[S,T]]$.
Since $\deg F^*(S)=\lam-1=\deg f(S,0)$ and $f(S,0)$ is monic, we obtain that
$F^*(S)$ is equal to $f(S,0)=\mathcal{F}_{\P}(S,0)g(S,0)$ up to multiplication by unit.
Hence, $g(S,T) \in \Zp[[S,T]]^\x$, which yields the claim.

Next, suppose that $L_K \subset \tilde{K}$.
Then it is known that $\tilde{K}$ coincides with the maximal abelian $p$-extension of $K$ which is unramified all primes outside above $p$.
Hence $X_{\Ki}/S \isom \Zp$.
Applying the snake lemma to (\ref{Ozaki01 ex seq (imag. quad. sp case)}),
% with respect to multiplication by $S$, 
we obtain an exact sequence
$\Zp \to X_{\tilde{K}}/(S,T) \to X_{\Ki}/S \to \Zp \to 0$.
Hence $X_{\tilde{K}}$ is $\Zp[[S,T]]$-cyclic and so $\mathfrak{X}_{\P}(\tilde{K})$ is.
This means that there are surjections 
$\Zp[[S,T]]/f(S,T) \surj \mathfrak{X}_{\P}(\tilde{K})$
and
$\Zp[[S]]/f(S,0) \surj X_{\tilde{K}}/T$.
So, in the same way as above, we obtain the claim.
\end{proof}

%%%%%%%%%%%%%%%%%%%%%%%%%%%%%%%%%%%%%%%%%%%%%%%%%%%%%%%%%%%%%%%%%%%%%%%%%%%%%%
%%%%%%%%%%%%%%%%%%%%%%%%%%%%%%%%%%%%%%%%%%%%%%%%%%%%%%%%%%%%%%%%%%%%%%%%%%%%%%
%%%%%%%%%%%%%%%%%%%%%%%%%%%%%%%%%%%%%%%%%%%%%%%%%%%%%%%%%%%%%%%%%%%%%%%%%%%%%%
%%%%%%%%%%%%%%%%%%%%%%%%%%%%%%%%%%%%%%%%%%%%%%%%%%%%%%%%%%%%%%%%%%%%%%%%%%%%%%
%%%%%%%%%%%%%%%%%%%%%%%%%%%%%%%%%%%%%%%%%%%%%%%%%%%%%%%%%%%%%%%%%%%%%%%%%%%%%%

\section{Proof of Theorem \ref{thm 3}}\label{Proof of Thm 3}
\subsection{A system of generators of $X_{\tilde{K}}$}\label{a system of generators}

First of all, we show a lemma from group theory.
Let $p$ be an arbitrary prime number and $G$ a finite abelian $p$-group such that $\dim_{\Fp}G/p=g$.
For $x \in G$ with order $p^n$, $\langle x \rangle_n$ denotes the cyclic group generated by $x$
(we also use the notation $\langle x \rangle$ instead of $\langle x \rangle_n$).

\begin{lem}\label{generators lem A}
With the notation as above, let $H$ be a subgroup of $G$ such that $G/H$ is cyclic, then
there exists a minimal system of generators $x_1,x_2 \ldots,x_g \in G$ such that the following holds:
{
\makeatletter
  \parsep   = 0pt
  \labelsep = .5pt
  \def\@listi{%
     \leftmargin = 20pt \rightmargin = 0pt
     \labelwidth\leftmargin \advance\labelwidth-\labelsep
     \topsep     = 0\baselineskip
     \partopsep  = 0pt \itemsep       = 0pt
     \itemindent = 0pt \listparindent = 10pt}
  \let\@listI\@listi
  \@listi
  \def\@listii{%
     \leftmargin = 20pt \rightmargin = 0pt
     \labelwidth\leftmargin \advance\labelwidth-\labelsep
     \topsep     = 0pt \partopsep     = 0pt \itemsep   = 0pt
     \itemindent = 0pt \listparindent = 10pt}
  \let\@listiii\@listii
  \let\@listiv\@listii
  \let\@listv\@listii
  \let\@listvi\@listii
  \makeatother
\ 
\begin{itemize}
\item[{\rm (i)}]
$G/H$ is generated by the image of $x_1$ under the projection. 
Moreover, $x_1$ has the minimum order among such elements.
\item[{\rm (ii)}]
$x_2,\ldots, x_g \in H$, if $g \ge 2$.
\end{itemize}
}
%Let $G$ be a finite abelian $p$-group and $A$ a nontrivial cyclic group.
%For any surjection
%$\varphi \colon G \to A$,
%there exists some cyclic subgroup $H$ of $G$ such that $H$ is a direct summand of $G$ and $\varphi(H)=A$.
\end{lem}

\begin{proof}
We denote the image of $x \in G$ in $G/H$ by $\overline{x}$.
First, we take a direct sum decomposition
% into direct sum
$$
G=
\bigop_{j=1}^{g_1}\langle x_{1j} \rangle_{n_{1j}}
\op
\bigop_{j=1}^{g_2}\langle x_{2j} \rangle_{n_{2j}}
\op
\cdots
\op
\bigop_{j=1}^{g_r}\langle x_{rj} \rangle_{n_{rj}}
\ \ 
(g_1+g_2+\cdots +g_r=g)
$$
with
$n_{i1} \le n_{i2} \le \ldots \le n_{ig_i}$ for each $1 \le i \le r$
and
\begin{eqnarray*}
G/H
&=&
\langle \overline{x_{11}} \rangle_{m_1} =\cdots= \langle \overline{x_{1g_1}} \rangle_{m_1}
\\
&\supsetneq&
\langle \overline{x_{21}} \rangle_{m_2} =\cdots= \langle \overline{x_{2g_2}} \rangle_{m_2}
\\
&\supsetneq&
\cdots
\\
&\supsetneq&
\langle \overline{x_{r1}} \rangle_{m_r} =\cdots= \langle \overline{x_{rg_r}} \rangle_{m_r}
=
\langle \overline{0} \rangle
\end{eqnarray*}
in $G/H$.
For simplicity, we put $x_i:=x_{i1}$ and $n_i:=n_{i1}$ $(1 \le i \le r)$.
For any $i$ and $j$ with $1 \le i \le r$ and $2 \le j \le g_i$,
there exist $h_{ij} \in H$ and $p \nmid a_{ij}$
such that
$
x_i=h_{ij}+a_{ij}x_{ij},
$
since $\langle \overline{x_{i}} \rangle_{m_i}=\langle \overline{x_{ij}} \rangle_{m_i}$.
%We see that $\ord(h_{ij}) \le \ord(x_{ij})$ by $\ord(x_{i}) \le \ord(x_{ij})$.
%Moreover, since $G$ is generated by $h_{ij}$ and the basis without $x_{ij}$,
Then, we can easily check that
$
\langle x_{i} \rangle_{n_{1}} \op \langle x_{ij} \rangle_{n_{ij}}
=
\langle x_{i} \rangle_{n_{1}} \op \langle h_{ij} \rangle_{n_{ij}}
$.
Hence, we may change $\langle x_{ij} \rangle_{n_{ij}}$ in the above decomposition with $\langle h_{ij} \rangle_{n_{ij}}$.
Therefore, by changing the names appropriately, we obtain a direct sum decomposition
% of $G$ into direct sums
$$
G=
\langle x_{1} \rangle_{n_{1}}
\op \cdots \op
\langle x_{r} \rangle_{n_{r}}
\op
\bigop_{h \in H'} \langle h \rangle
$$
for some subset $H'$ in $H$.
Next, for any $i$ with $2 \le i \le r$, 
there exist $h_{i} \in H$ and $p \mid a_i$
such that
$
x_i=h_{i}+a_{i}x_{1},
$
since $\langle \overline{x_{1}} \rangle_{m_1} \supsetneq \langle \overline{x_{i}} \rangle_{m_i}$.
%By $m_1>m_i$, we obtain $h_i \equiv x_i$ mod $pG$.
Then
$$
G=
\langle x_{1} \rangle
+
\langle h_2,\ldots,h_r \rangle
+
\bigop_{h \in H'} \langle h \rangle
$$
and elements in $\bigcup_{i=2}^r \{ h_i\} \cup H'$ are linearly independent mod $pG$.
Finally, we show that the above $x_1$ can be taken such that it has the minimum order among the elements whose images generate $G/H$.
Let $x_1' \in G$ be such an element.
Then we have $x_1'=h'+a'x_1$ for some $h' \in H$ and $p \nmid a'$.
We also obtain 
$h'=b_1x_1+\sum_{i=2}^rb_ih_i+\sum_{h \in H'}b_h h$
for some $b_1, \ldots b_r, h_h \in \Z$.
Then $b_1x_1 \in \langle x_{1} \rangle \cap H$, so that $p^{m_1} \mid b_1$.
Since $x_1'$ has the form $x_1'=(a'+b_1)x_1+\sum_{i=2}^rb_ih_i+\sum_{h \in H'}b_h h$,
we obtain that
$
G=
\langle x_{1}' \rangle
+
\langle h_2,\ldots,h_r \rangle
+
\bigop_{h \in H'} \langle h \rangle
$.
%Choose any basis $x_1,\ldots,x_g$ of $G$.
%Assume that $\varphi(x_i)$ does not generate $A$ for any $i$'s.
%Then $\varphi$ is not surjective.
%This is a contradiction.
%Therefore, there exists some $i$ such that $\varphi(x_i)$ generate $A$.
%Then the cyclic subgroup generated by $x_i$ satisfies the condition.
\end{proof}

%\begin{lem}\label{generators lem B}
%Suppose that $G$, $H$, and $x_1,\ldots, x_g$ satisfy the condition in Lemma \ref{generators lem A}.
%Also, suppose that $H$ contains a direct summand $H'$ of $G$ such that $G/H'$ is cyclic.
%Then $G=\langle x_1 \rangle \op H'$.
%In particular, the orders of $x_1$ and $H'$ are uniquely determined.
%\end{lem}

Let $p$ be an odd prime number, $K$ an imaginary quadratic field such that $p$ does not split.
We use the notation in \S \ref{notation} for such $p$ and $K$, $\tilde{K}$, $\Ki$, $X_{\tilde{K}}$, etc.
Put $g:=\dim_{\Fp} (A_K/p)$.
We use Ozaki's exact sequence in the case where $p$ does not split:

\begin{lem}{\rm (\cite[Lemma 1]{Ozaki01})}\label{Ozaki01 ex seq (non-sp case)}
There is an exact sequence of $\Zp [[S]]$-modules
$$
0 \to X_{\tilde{K}}/T \to X_{\Ki} \to \Gal(L_{\Ki} \cap \tilde{K}/\Ki) \to 0.
$$
\end{lem}

\begin{lem}\label{an isom1}
There is a canonical isomorphism
$$
\Gal(L_{\Ki} \cap \tilde{K}/\Ki) \isom \Gal(L_K \cap \tilde{K}/K).
$$
\end{lem}

\begin{proof}
Let $L'$ be the maximal abelian subextension in $L_{\Ki}/K$.
We see that $L'=\Ki L_K$, 
since $L_K$ is the fixed field by the inertia group of a prime lying above $p$ in $\Gal(L'/K)$ and $L'/\Ki L_K$ is unramified.
If we show that
\begin{eqnarray}\label{an isom2}
L_{\Ki} \cap \tilde{K} =\Ki(L_K \cap \tilde{K}),
\end{eqnarray}
then we obtain
%\begin{eqnarray*}
$\Gal(L_{\Ki} \cap \tilde{K}/\Ki)
=
\Gal(\Ki(L_K \cap \tilde{K})/\Ki)
%\\
%&\isom&
%\Gal(L_K \cap \tilde{K}/(L_K \cap \tilde{K}) \cap \Ki)
\isom
\Gal(L_K \cap \tilde{K}/K).
$
%\end{eqnarray*}
Let us show (\ref{an isom2}).
We have 
$L_K \cap \tilde{K} \subset L_K \subset \Ki L_K \subset L_{\Ki}$, so that 
\begin{eqnarray}\label{subset1}
\Ki (L_K \cap \tilde{K}) \subset L_{\Ki} \cap \tilde{K}.
\end{eqnarray}
Since $L_{\Ki} \cap \tilde{K}$ is abelian over $K$ and unramified over $\Ki$, 
we see that
$L_{\Ki} \cap \tilde{K} \subset L'=\Ki L_K$.
This yields 
\begin{eqnarray}\label{subset2}
L_{\Ki} \cap \tilde{K} \subset (\Ki L_K) \cap \tilde{K} \subset \Ki (L_K \cap \tilde{K}).
\end{eqnarray}
The claim follows from (\ref{subset1}) and (\ref{subset2}).
\end{proof}

Since $X_{\Ki}/(p,S) \isom A_K/p$,
\begin{eqnarray*}\label{ineq of generators}
g-1 \le \dim_{\Fp} X_{\tilde{K}}/(p,S,T) \le g+1
\end{eqnarray*}
by Lemmas \ref{Ozaki01 ex seq (non-sp case)} and \ref{an isom1}.
Put 
$$
p^m:=\# \Gal(L_K \cap \tilde{K}/K).
$$
Now, we show Theorem \ref{thm 3} (i).
If $m=0$, i.e., $L_K \cap \tilde{K}=K$, then 
$X_{\tilde{K}}/T \isom X_{\Ki}$,
so that
$
\dim_{\Fp} (X_{\tilde{K}}/(p,S,T))
=g.
$
This completes the proof of Theorem \ref{thm 3} (i).
In the following, we assume that $m>0$.

We fix certain generators $x_1,\ldots, x_g$ of $X_{\Ki}$ as a $\Zp[[S]]$-module as follows.
%If $m=0$, we choose any generators $x_1,\ldots,x_g$ of $X_{\Ki}$.
%Suppose that $m>0$.
Applying Lemma \ref{generators lem A} to $\Gal(L_K/K)$ and its quotient $\Gal(L_K \cap \tilde{K}/K)$,
we can choose the basis of $\Gal(L_K/K)$ which satisfies the conditions in Lemma \ref{generators lem A}.
Moreover, 
%
%we have a direct summand $H$ of $\Gal(L_K/K)$ with order $p^{n_1}$.
%Then we have another direct summand $H'$ of $\Gal(L_K/K)$ such that the fixed field contains $L_K \cap \tilde{K}$.
%We fix {\it the} $H'$.
%
applying Nakayama's lemma to this basis,
%of $\Gal(L_K/K) =H \op H'$, 
then we can choose $x_1,\ldots, x_g$ in $X_{\Ki}$ which satisfy the conditions bellow.
{
\makeatletter
  \parsep   = 0pt
  \labelsep = .5pt
  \def\@listi{%
     \leftmargin = 20pt \rightmargin = 0pt
     \labelwidth\leftmargin \advance\labelwidth-\labelsep
     \topsep     = 0\baselineskip
     \partopsep  = 0pt \itemsep       = 0pt
     \itemindent = 0pt \listparindent = 10pt}
  \let\@listI\@listi
  \@listi
  \def\@listii{%
     \leftmargin = 20pt \rightmargin = 0pt
     \labelwidth\leftmargin \advance\labelwidth-\labelsep
     \topsep     = 0pt \partopsep     = 0pt \itemsep   = 0pt
     \itemindent = 0pt \listparindent = 10pt}
  \let\@listiii\@listii
  \let\@listiv\@listii
  \let\@listv\@listii
  \let\@listvi\@listii
  \makeatother
\ 
\begin{itemize}
\item
$x_1,\ldots, x_g$ generate $X_{\Ki}$.
\item
The image of $x_1$ generates $\Gal(L_K \cap \tilde{K}/K)$.
Also the images of $x_2,\ldots, x_g$
% generate $H'$, in particular, these 
become $0$ in $\Gal(L_K \cap \tilde{K}/K)$.
\item
Moreover, if 
$\Gal(L_K/K) \isom \Gal(L_K \cap \tilde{K}/K) \op \Gal(L_K/L_K \cap \tilde{K})$ 
and the exponent of $\Gal(L_K/L_K \cap \tilde{K})$ is equal to or less than $p^m$, then
$x_1$ can be replaced modulo ($x_2,\ldots, x_g)\Zp[[S]]$
(this fact is useful in \S \ref{Classification in the case lam=2}).
\end{itemize}
}

\noindent
Note that $x_1,\ldots, x_g$ are defined modulo $SX_{\Ki}$.
%Note that $\Gal(\Ki/K)$ acts on $\Gal(L_K \cap \tilde{K}/K)$ trivially.
In the following, we assign the sum ``$\sum_{j=2}^g$'' to $0$ if $g=1$.
For any $a_1,\ldots,a_m \in X_{\Ki}$, $(a_1, \ldots, a_m)_X$ denotes the $\Zp[[S]]$-submodule of $X_{\Ki}$ generated by $a_1, \ldots, a_m$.
Put
\begin{eqnarray*}
&&
M:=(p^m x_1, Sx_1)_X,
\ \ 
N:=
\begin{cases}
(x_2,\ldots,x_g)_X & \text{if $g>1$},
\\
0 & \text{if $g=1$}, 
\end{cases}
\\
&&
\ \ 
\text{and}
\ \ 
%\\
%&&
L:=(p,S)(M+N)
=\left(p^{m+1}x_1, \; pSx_1, \; S^2x_1, \; px_j, \; Sx_j \; |\; j=2,\ldots,g\right)_X.
%(p,S)(p^m x_1, Sx_1, x_2,\ldots,x_g)_X
\end{eqnarray*}
By Lemma \ref{Ozaki01 ex seq (non-sp case)}, we identify $X_{\tilde{K}}/T$ 
with
 the submodule in $X_{\Ki}$.

\begin{prop}\label{key3}
\makeatletter
  \parsep   = 0pt
  \labelsep = .5pt
  \def\@listi{%
     \leftmargin = 20pt \rightmargin = 0pt
     \labelwidth\leftmargin \advance\labelwidth-\labelsep
     \topsep     = 0\baselineskip
     \partopsep  = 0pt \itemsep       = 0pt
     \itemindent = 0pt \listparindent = 10pt}
  \let\@listI\@listi
  \@listi
  \def\@listii{%
     \leftmargin = 20pt \rightmargin = 0pt
     \labelwidth\leftmargin \advance\labelwidth-\labelsep
     \topsep     = 0pt \partopsep     = 0pt \itemsep   = 0pt
     \itemindent = 0pt \listparindent = 10pt}
  \let\@listiii\@listii
  \let\@listiv\@listii
  \let\@listv\@listii
  \let\@listvi\@listii
  \makeatother
\ 
We have the following.
\begin{itemize}
\item[{\rm (i)}]
$X_{\tilde{K}}/T = M+N$.
\item[{\rm (ii)}]
$(M+L/L) \cap (N+L/L) =0$, and hence $X_{\tilde{K}}/(p,S,T) \isom (M+L/L) \op (N+L/L)$.
\item[{\rm (iii)}]
$N+L/L \isom \Fp^{\op (g-1)}$.
\end{itemize}
\end{prop}

\begin{proof}
(i)
By 
Lemmas \ref{Ozaki01 ex seq (non-sp case)} and \ref{an isom1}, 
we see 
$x_2,\ldots,x_g \in X_{\tilde{K}}/T$, and also
$$
p^m x_1, Sx_1 \in X_{\tilde{K}}/T,
$$
since $\Gal(L_K \cap \tilde{K}/K)$ has order $p^m$ and 
the
trivial action by $\Gal(\Ki/K)$.
Hence 
$
M+N \subset X_{\tilde{K}}/T.
$
Since $[X_{\Ki}:M+N] \le p^m$, we have the claim.
\\
(ii)
Any elements in $(M+L/L) \cap (N+L/L)$ are represented 
by some elements in $X_{\Ki}$ of the form 
$$
a_1 p^mx_1+a_2Sx_1=\sum_{j=2}^g b_j x_j +c 
%\in X_{\Ki}
\ \ 
(a_1,a_2,b_2,\ldots, b_g \in \{0, \ldots, p-1\}, c \in L).
$$
The left hand side is congruent to $0$ modulo $(p,S)X_{\Ki}$.
Since $x_1, \ldots, x_g$ are linearly independent in $A_K/p$, we see 
$b_2,\ldots, b_g=0$.
This implies that $(M+L/L) \cap (N+L/L) =0$.
\\
(iii)
By an exact sequence
$N/(p,S) \to X_{\Ki}/(p,S) \to (X_{\Ki}/N)/(p,S) \to 0$,
we have
$N/(p,S) \isom (\Zp/p)^{\op (g-1)}$.
Therefore we have only to show $N \cap L \subset (p,S)N$, since
$[N:N \cap L] \le p^{g-1}$.
Any elements in $N \cap L$ are represented 
by some elements in $X_{\Ki}$ of the form 
%Each element in $N \cap L$ is written as 
$$
(A_1 p^{m+1}+A_2pS+A_3S^2)x_1
+
\sum_{j=2}^g ( B_jp+B'_jS) x_j
=
\sum_{j=2}^g B''_j x_j
\ \ 
(A_1,A_2,A_3,B_j,B'_j, B''_j \in \Zp[[S]]).
$$
Since $x_1, \ldots, x_g$ are linearly independent modulo $(p,S)X_{\Ki}$, we see 
$B''_2,\ldots, B''_g \in (p,S)\Zp[[S]]$.
This implies that $N \cap L \subset (p,S)N$.
\end{proof}

We denote the image of $x \in X_{\Ki}$ in $\Gal(L_K/K)$ by $\overline{x}$.
The following lemma gives criterions in the case where $\langle \overline{x_1} \rangle$ is a direct summand of $\Gal(L_K/K)$.
We will use the lemma in \S \ref{The case A_K cyclic} and \S \ref{Classification in the case lam=2}.

\begin{lem}\label{Sx_1 in K, p^mx_1 notin K}
Denote the order of $\overline{x_1}$ by $\ord(\overline{x_1})$.
\makeatletter
  \parsep   = 0pt
  \labelsep = .5pt
  \def\@listi{%
     \leftmargin = 20pt \rightmargin = 0pt
     \labelwidth\leftmargin \advance\labelwidth-\labelsep
     \topsep     = 0\baselineskip
     \partopsep  = 0pt \itemsep       = 0pt
     \itemindent = 0pt \listparindent = 10pt}
  \let\@listI\@listi
  \@listi
  \def\@listii{%
     \leftmargin = 20pt \rightmargin = 0pt
     \labelwidth\leftmargin \advance\labelwidth-\labelsep
     \topsep     = 0pt \partopsep     = 0pt \itemsep   = 0pt
     \itemindent = 0pt \listparindent = 10pt}
  \let\@listiii\@listii
  \let\@listiv\@listii
  \let\@listv\@listii
  \let\@listvi\@listii
  \makeatother
\ 
\begin{itemize}
\item[{\rm (i)}]
If $p^m=\ord(\overline{x_1})$,
in other words, 
$\Gal(L_K/K) \isom \Gal(L_K \cap \tilde{K}/K) \op \Gal(L_K/L_K \cap \tilde{K})$,
then
$$
\dim_{\Fp}(X_{\tilde{K}}/(p,S,T))
=
\begin{cases}
g-1 & \text{if $Sx_1 \in (p,S)N$},
\\
g & \text{otherwise}.
\end{cases}
$$
\item[{\rm (ii)}]
Suppose that 
$
\Gal(L_K/K) 
=
\langle \overline{x_1} \rangle \op \langle \overline{x_2}, \ldots, \overline{x_g} \rangle
$
and 
$0<p^m<\ord(\overline{x_1})$, then
$p^mx_1 \not\in L
$.
In particular, $\dim_{\Fp}(X_{\tilde{K}}/(p,S,T)) \ge g$.
\end{itemize}
\end{lem}

\begin{proof}
(i)
By the assumption, we have $p^m x_1 \equiv 0$ mod $SX_{\Ki}$.
Therefore, there exist some $A_1,\ldots,A_g \in \Zp[[S]]$ such that
$$
p^m x_1=A_1 S x_1+\sum_{j=2}^g A_j S x_j
\ \ 
\text{in $X_{\Ki}$}.
$$
The second term in the right hand side is in $(p,S)N \subset 
L
$.
This implies that $M+
L/L
$ is generated by $Sx_1+L$.
Moreover, we see that $p^{m+1}x_1$ is written as a linear form of $pSx_1$ and elements in $(p,S)N$.
We claim that $Sx_1 \in 
L
$ if and only if $Sx_1 \in (p,S)N$,
which implies the conclusion by Proposition \ref{key3}.
Assume that $Sx_1 \in 
L
$.
Then there exist some $A, A', B_j, B'_j \in \Zp[[S]]$ ($j=2,\ldots,g$) such that
$$
Sx_1 =(ApS+A'S^2)x_1+\sum_{j=2}^g (B_j p+B'_jS)x_j,
$$
so that 
$$
(1-Ap-A'S)Sx_1 =\sum_{j=2}^g (B_j p+B'S)x_j.
$$
This implies that $Sx_1 \in (p,S)N$ since $1-Ap-A'S \in \Zp[[S]]^\x$.
The converse is trivial.
\\
(ii)
Assume that $p^mx_1 \in L$.
Then, in a similar way as (i), we see that there exists some $a \in \Zp$ such that
$$
p^m x_1 \equiv a p^{m+1} x_1 \ \ \text{mod $SX_{\Ki}$},
$$
since 
$
\Gal(L_K/K) 
=
\langle \overline{x_1} \rangle \op \langle \overline{x_2}, \ldots, \overline{x_g} \rangle
$.
This implies that $p^m x_1$ equals $0$ in $A_K$, since $1-ap \in \Zp^\x$, which is a contradiction.
\end{proof}

%%%%%%%%%%%%%%%%%%%%%%%%%%%%%%%%%%%%%%%%%%%%%%%%%%%%%%%%%%%%%%%%%%%%%%%%%%%%%%%%%%%%%%%%%%%%%%%%

\subsection{Classification in the case where $A_K$ is cyclic}\label{The case A_K cyclic}

Now, we show Theorem \ref{thm 3} (ii).
Suppose that $L_K \cap \tilde{K} \neq K$ and that $\dim_{\Fp} (A_K/p)=1$, in other words, $g=1$.
Then, $N=0$ and there is an isomorphism 
$
X_{\Ki} \isom \Zp[[S]]/F(S),
$
where $F(S) \in\Zp[S]$ is 
the distinguished polynomial generating the characteristic ideal
 of $X_{\Ki}$,
since $A_K$ is cyclic and $X_{\Ki}$ has no nontrivial finite $\Zp[[S]]$-submodules.
Note that we have $Sx_1 \neq 0$, since $S$ dose not divide $F(S)$.
In this case, we can apply Lemma \ref{Sx_1 in K, p^mx_1 notin K}.
\\[2mm]
{\it Proof of {\rm (ii-a)}.}
Assume that $\lam=1$.
If $p^m=\# A_K$, then 
%$Sx_1 \neq 0$ in $X_{\Ki}$, since $S$ dose not divide $F(S)$.
%Hence $Sx_1 \not\in K$ 
%This implies that 
$\dim_{\Fp}(X_{\tilde{K}}/(p,S,T))=1$ by Lemma \ref{Sx_1 in K, p^mx_1 notin K} (i).
On the other hand, assume that $0<p^m<\# A_K$.
%Then we see that the characteristic polynomial of $X_{\Ki}$ annihilates $x_1$.
Since 
$F(S)x_1=0$ in $X_{\Ki}$ 
and $F(0)x_1 \in 
L
$, we obtain $S x_1 \in 
L
$, which implies that $\dim_{\Fp}(X_{\tilde{K}}/(p,S,T))=1$ by 
Proposition \ref{key3} (ii), (iii).
%Lemma \ref{Sx_1 in K, p^mx_1 notin K} (ii).
\\[2mm]
{\it Proof of {\rm (ii-b)}.}
Assume that $\lam \ge 2$.
If $p^m=\# A_K$, then 
%in the same way as (ii-a), we obtain $Sx_1 \neq 0$ in $X_{\Ki}$, which induces that 
$\dim_{\Fp}(X_{\tilde{K}}/(p,S,T))=1$.
On the other hand, assume that $0<p^m<\# A_K$.
Since we see that $F(S) \equiv F(0)$ mod $(pS, S^2)$ from $\deg F(S) \ge 2$, we obtain
\begin{eqnarray*}
X_{\tilde{K}}/(p,S,T)
&\isom&
(p^m,S, F(S)) \Zp[[S]] /(p^{m+1},pS, S^2, F(S))\Zp[[S]]
\\
&=&
(p^m,S)\Zp[[S]]/(p^{m+1},pS, S^2)\Zp[[S]]
\\
&\isom& 
(\Z/p)^{\op 2}.
\end{eqnarray*}
This completes the proof of (ii).

%%%%%%%%%%%%%%%%%%%%%%%%%%%%%%%%%%%%%%%%%%%%%%%%%%%%%%%%%%%%%%%%%%%%%%%%%%%%%%
%%%%%%%%%%%%%%%%%%%%%%%%%%%%%%%%%%%%%%%%%%%%%%%%%%%%%%%%%%%%%%%%%%%%%%%%%%%%%%
%%%%%%%%%%%%%%%%%%%%%%%%%%%%%%%%%%%%%%%%%%%%%%%%%%%%%%%%%%%%%%%%%%%%%%%%%%%%%%
%%%%%%%%%%%%%%%%%%%%%%%%%%%%%%%%%%%%%%%%%%%%%%%%%%%%%%%%%%%%%%%%%%%%%%%%%%%%%%
%%%%%%%%%%%%%%%%%%%%%%%%%%%%%%%%%%%%%%%%%%%%%%%%%%%%%%%%%%%%%%%%%%%%%%%%%%%%%%

\section{Conditions of $X_{\tilde{K}}$ to be cyclic in the case where $\lam=2$}\label{Classification in the case lam=2}
\subsection{Setting and Methods}\label{Methods}
Let $p$ be an odd prime number, $K$ an imaginary quadratic field such that $p$ does not split.
We use the notation in \S \ref{notation} for such $p$ and $K$.
We consider conditions that
 $X_{\tilde{K}}$ 
becomes
 $\Zp[[S,T]]$-cyclic.
%By Proposition \ref{key3} and Lemma \ref{Sx_1 in K, p^mx_1 notin K}, the rest case which we consider is 
%$\dim_{\Fp}(A_K/p)=2$ and $\Gal(L_K \cap \tilde{K}/K)$ is a direct summand of $\Gal(L_K/K) \isom A_K$.
%We assume them.
In the case where $\dim_{\Fp}(A_K/p)=1$, Theorem \ref{thm 3} gives the condition, so that we have only to consider the case where $\dim_{\Fp}(A_K/p)=2$.

In this section, we treat the case where 
$\dim_{\Fp}(A_K/p)=2$ and $\Gal(L_K \cap \tilde{K}/K)$ is a direct summand of $\Gal(L_K/K) \isom A_K$.
Moreover, we add more assumptions that $\lam=2$ and that
 the roots $\alp, \beta \in \overline{\Qp}$ of the distinguished polynomial generating the characteristic ideal 
of $X_{\Ki}$ satisfy $\alp \neq \beta$.
We remark that the latter part of this assumption is expected to be held.

Define $\mathcal{O}:=\Zp[\alp, \beta]$ and $\Lam:=\mathcal{O}[[S]]$, and let
$\pi$
be a uniformizer in $\mathcal{O}$.
Then, by the above assumption, the characteristic ideal 
%${\rm char}_{\Zp[[S]]}(X_{\Ki})\Zp[[S]]$ 
of $X_{\Ki}\ox_{\Zp} \mathcal{O}$ is described as
$$
%{\rm char}_{\Zp[[S]]}(X_{\Ki})\Zp[[S]] \ox_{\Zp} \mathcal{O}= 
(S-\alp)(S-\beta)\Lam
\ \ 
(\alp, \beta \in \pi \mathcal{O} \setminus \{ 0 \},\ \alp \neq \beta).
$$
Then, by Koike \cite{Koike}, there exist an integer $k$ with $0 \le k \le {\rm ord}_{\pi} (\beta-\alp)$
and an $\mathcal{O}$-basis $\mathbf{e}_1,\mathbf{e}_2$ of $X_{\Ki}\ox_{\Zp} \mathcal{O}$ such that the homomorphism of $\Lam$-modules
\begin{eqnarray}\label{injection to elementary module}
X_{\Ki}\ox_{\Zp} \mathcal{O}
\inj 
\Lam/(S-\alp) \op \Lam/(S-\beta);
\ \ 
\mathbf{e}_1 \ \mapsto 
\begin{bmatrix}
1
\\
1
\end{bmatrix},
\ \ 
\mathbf{e}_2 \ \mapsto 
\begin{bmatrix}
0
\\
\pi^k
\end{bmatrix}
\end{eqnarray}
is injective.
Note that $k$ depends only on the isomorphism class of $X_{\Ki}$.
We regard $X_{\Ki}\ox_{\Zp} \mathcal{O}$ 
as
 a $\Lam$-submodule of $\Lam/(S-\alp) \op \Lam/(S-\beta)$ by the above injection.
We can express the action of $S$ by
$$
S
\begin{bmatrix}
1
\\
1
\end{bmatrix}
=
\begin{bmatrix}
\alp
\\
\beta
\end{bmatrix},
\ \ 
S
\begin{bmatrix}
0
\\
\pi^k
\end{bmatrix}
=
\begin{bmatrix}
0
\\
\beta \pi^k
\end{bmatrix}.
$$
For convenience, we regard $X_{\Ki} \subset X_{\Ki}\ox_{\Zp} \mathcal{O}$ by the injection $x \mapsto x \ox 1$, and
put 
$\gam :=(\beta-\alp)\pi^{-k}$ and ${\rm ord}(a):={\rm ord}_{\pi}(a)$ for $a \in \mathcal{O}$.
We take such generators $x_1,x_2$ of $X_{\Ki}$ as in \S \ref{a system of generators}.
They are represented as
\begin{eqnarray*}\label{rep}
x_1
&=&
\lam_{11}\mathbf{e}_1+\lam_{12}\mathbf{e}_2,
\\
x_2
&=&
\lam_{21}\mathbf{e}_1+\lam_{22}\mathbf{e}_2
\end{eqnarray*}
for some $\lam_{ij} \in \mathcal{O}$.
We have $A_K \ox_{\Zp} \mathcal{O} = (\mathcal{O}/\pi^{N_1})x_1 \op (\mathcal{O}/\pi^{N_2})x_2$ and
$\Gal(L_K \cap \tilde{K}/K)\ox_{\Zp} \mathcal{O}= (\mathcal{O}/\pi^{N_1})x_1$
for some $N_1,N_2 \in \Z$.
%(Note that, using the notation $n_1$ which defined in \S \ref{a system of generators}, if $\mathcal{O}/\Zp$ is unramified then $n_1=N_1$, and otherwise $2n_1=N_1$.)
We denote 
by 
$\langle a_1,\ldots ,a_l\rangle_{\Lam}$
(resp. $\langle a_1,\ldots a_l \rangle_{\mathcal{O}}$)
the $\Lam$-submodule (resp. $\mathcal{O}$-submodule) in $\Lam/(S-\alp) \op \Lam/(S-\beta)$ generated by $a_1,\ldots,a_l$.

%%%%%%%%%%%%%%%%%%%%%%%%%%%%%%%%%%%%%%%%%%%%%%%%%%%%%%%%%%

\begin{lem}\label{A_K ox_{Zp} mathcal{O}}
We have an isomorphism
\begin{eqnarray*}
A_K \ox_{\Zp} \mathcal{O}
\isom
\begin{cases}
\mathcal{O}/\alp \op \mathcal{O}/\beta
&
\text{if $\ord(\gam) \ge \min\{ \ord(\alp),\ord(\beta) \}$},
\\
\mathcal{O}/\gam \op \mathcal{O}/\f{\alp\beta}{\gam}
&
\text{if $\ord(\gam) < \min\{ \ord(\alp),\ord(\beta) \}$}.
%\text{otherwise}.
\end{cases}
\end{eqnarray*}
\end{lem}

\begin{proof}
We have
\begin{eqnarray*}
&&
S(X_{\Ki}\ox_{\Zp} \mathcal{O})
\\[1mm]
&&
=
\left\langle
\begin{bmatrix}
\alp
\\
\beta
\end{bmatrix}
,\ 
\begin{bmatrix}
0
\\
\beta \pi^k
\end{bmatrix}
\right\rangle_{\mathcal{O}}
=
\left\langle
\alp
\mathbf{e}_1+
\gam
\mathbf{e}_2
,\,
\beta
\mathbf{e}_2
\right\rangle_{\mathcal{O}}
\\[1mm]
&&
=
\begin{cases}
\left\langle
\alp \mathbf{e}_1
,\ 
\beta \mathbf{e}_2
\right\rangle_{\mathcal{O}}
&
\text{if $\ord(\gam) \ge \min\{ \ord(\alp),\ord(\beta) \}$},
\\[1mm]
%%
%\left\langle
%\alp \mathbf{e}_2+\gam \mathbf{e}_2
%,\ 
%-\f{\alp \beta}{\gam}\mathbf{e}_1
%\right\rangle_{\Lam}
%=
\left\langle
\gam \left(\f{\alp}{\gam}\mathbf{e}_1+\mathbf{e}_2\right)
,\ 
\f{\alp \beta}{\gam}\mathbf{e}_1
\right\rangle_{\mathcal{O}}
&
\text{if $\ord(\gam) < \min\{ \ord(\alp),\ord(\beta) \}$}.
\end{cases}
\end{eqnarray*}
Since $A_K \ox_{\Zp} \mathcal{O} \isom (X_{\Ki}\ox_{\Zp} \mathcal{O})/S$, this yields the claim.
\end{proof}

\begin{lem}\label{pi^{pi^N}x_1}
The following conditions hold:
\begin{eqnarray*}
\begin{matrix}
\begin{cases}
\ 
\lam_{11}\f{\pi^{N_1}}{\alp} \in \mathcal{O},
\\[3mm]
\ 
\f{\pi^{N_1}}{\beta}\left( \lam_{12}-\lam_{11}\f{\gam}{\alp}\right) \in \mathcal{O},
\end{cases}
&&&&
\begin{cases}
\ 
\lam_{21}\f{\pi^{N_2}}{\alp} \in \mathcal{O},
\\[3mm]
\ 
\f{\pi^{N_2}}{\beta}\left( \lam_{22}-\lam_{21}\f{\gam}{\alp}\right) \in \mathcal{O}.
\end{cases}
\end{matrix}
\end{eqnarray*}
\end{lem}

\begin{proof}
Since $\pi^{N_1}x_1 \in SX_{\Ki} \ox_{\Zp}\mathcal{O}$, we have 
$\pi^{N_1}(\lam_{11}\mathbf{e}_1+\lam_{12}\mathbf{e}_2) \in \langle \alp\mathbf{e}_1+\gam\mathbf{e}_2,\ \beta\mathbf{e}_2\rangle_{\mathcal{O}}$.
Hence, there exist some $s,t \in \mathcal{O}$ such that $\lam_{11}\pi^{N_1}=s \alp$, $\lam_{12}\pi^{
N_1}
=s \gam+t\beta$.
They induce the first relations.
The rest are verified in the same way.
\end{proof}

Recall that we defined $N$ as the $\Zp[[S]]$-submodule of $X_{\Ki}$ generated by $x_2$ and that Lemma \ref{Sx_1 in K, p^mx_1 notin K} (i) gives a criterion whether $X_{\tilde{K}}$ is $\Zp[[S,T]]$-cyclic or not.
In the same way as the proofs of Proposition \ref{key3} and Lemma \ref{Sx_1 in K, p^mx_1 notin K}, we can show the following

\begin{lem}\label{pre-lem_criterion}
$X_{\tilde{K}}$ is $\Zp[[S,T]]$-cyclic if and only if $S x_1 \in (\pi,S)N \ox_{\Zp} \mathcal{O}$.
\end{lem}

\begin{proof}
Note that there is nothing to show if $\mathcal{O}/\Zp$ is unramified.
And also, note that 
$X_{\tilde{K}}$ is $\Zp[[S,T]]$-cyclic if and only if $X_{\tilde{K}}\ox_{\Zp} \mathcal{O}$ is $\mathcal{O}[[S,T]]$-cyclic by Nakayama's lemma. 
Put
$L':=(\pi,S)(M \ox_{\Zp} \mathcal{O}+N \ox_{\Zp} \mathcal{O})$.
Then, in the same way as the proofs of Proposition \ref{key3} (ii) and (iii), we can show that
$X_{\tilde{K}} \ox_{\Zp} \mathcal{O} \isom (M \ox_{\Zp} \mathcal{O}+L'/L') \op (N \ox_{\Zp} \mathcal{O}+L'/L')$
and
$(N \ox_{\Zp} \mathcal{O}+L')/L' \isom \Fp$.
Moreover, in the same way as the proof of Lemma \ref{Sx_1 in K, p^mx_1 notin K} (i), we can show that
$Sx_1 \in (\pi,S)N \ox_{\Zp} \mathcal{O}$
if and only if 
$Sx_1 \in L'$.
\end{proof}

\begin{lem}\label{lem_criterion}
$X_{\tilde{K}}$ is $\Zp[[S,T]]$-cyclic if and only if
%The condition that
%$Sx_1 \in (\pi, S)N \ox_{\Zp}\mathcal{O}=\langle \pi x_2, Sx_2 \rangle_{\Lam}$
%holds if and only if 
there exist some $f(S),g(S) \in \Lam$ such that
\begin{eqnarray}\label{criterion}
\begin{cases}
\ 
\lam_{11}\alp
=
f(\alp)\lam_{21}\pi,
\\
\ 
(\lam_{11}+\lam_{12}\pi^k)\beta
=
(\lam_{21}+\lam_{22}\pi^k)(f(\beta)\pi+g(\beta)\gam\pi^k).
\end{cases}
\end{eqnarray}
\end{lem}

\begin{proof}
We have
$Sx_1=
\begin{bmatrix}
\lam_{11}\alp
\\
\lam_{11}\beta+\lam_{12}\beta\pi^k
\end{bmatrix}
$
and
\begin{eqnarray*}
(\pi, S)N \ox_{\Zp}\mathcal{O}
%\left\langle
%\lam_{21}\pi\mathbf{e}_1+\lam_{22}\pi\mathbf{e}_2,
%\ 
%\lam_{21}\alp\mathbf{e}_1+(\lam_{21}\gam+\lam_{22}\beta)\mathbf{e}_2
%\right\rangle_{\Lam}
%\\
%&=&
%\left\langle
%\lam_{21}\pi\mathbf{e}_1+\lam_{22}\pi\mathbf{e}_2,
%\ 
%(\lam_{21}+\lam_{22}\pi^k)\gam\mathbf{e}_2
%\right\rangle_{\Lam}
&=&
\left\langle
\begin{bmatrix}
\lam_{21}\pi
\\
(\lam_{21}+\lam_{22}\pi^{k})\pi
\end{bmatrix},
\ 
\begin{bmatrix}
\lam_{21}\alp
\\
\lam_{21}\beta+\lam_{22}\beta\pi^k
\end{bmatrix}
\right\rangle_{\Lam}
\\
&=&
\left\langle
\begin{bmatrix}
\lam_{21}\pi
\\
(\lam_{21}+\lam_{22}\pi^{k})\pi
\end{bmatrix},
\ 
\begin{bmatrix}
0
\\
(\lam_{21}+\lam_{22}\pi^k)\gam\pi^k
\end{bmatrix}
\right\rangle_{\Lam}.
\end{eqnarray*}
Then 
the
claim follows immediately by Lemma \ref{pre-lem_criterion}.
\end{proof}

Now, we consider the case 
where
$N_1 \ge N_2$.
Then, as we mentioned in \S \ref{a system of generators},
%since $x_2$ is fixed and 
$x_1$ is defined modulo $(x_2,SX_{\Ki}) \subset X_{\Ki}$,
so that we may change $x_1$ up to 
$x_2\Lam$.
Furthermore, since $\lam_{11}\lam_{22}-\lam_{12}\lam_{21} \in \mathcal{O}^\x$,
at least one (possibly both) of $\lam_{21}$, $\lam_{22}$ is in $\mathcal{O}^\x$. 
When $\lam_{21} \in \mathcal{O}^\x$
(resp. $\lam_{22} \in \mathcal{O}^\x$), 
we may assume that $\lam_{21}=1$, $\lam_{11}=0$, and $\lam_{12}=1$ 
(resp. $\lam_{22}=1$, $\lam_{12}=0$, and $\lam_{11}=1$).
Thus, if $N_1 \ge N_2$, we are reduced to only two cases:
{
\makeatletter
  \parsep   = 0pt
  \labelsep = .5pt
  \def\@listi{%
     \leftmargin = 20pt \rightmargin = 0pt
     \labelwidth\leftmargin \advance\labelwidth-\labelsep
     \topsep     = 0\baselineskip
     \partopsep  = 0pt \itemsep       = 0pt
     \itemindent = 0pt \listparindent = 10pt}
  \let\@listI\@listi
  \@listi
  \def\@listii{%
     \leftmargin = 20pt \rightmargin = 0pt
     \labelwidth\leftmargin \advance\labelwidth-\labelsep
     \topsep     = 0pt \partopsep     = 0pt \itemsep   = 0pt
     \itemindent = 0pt \listparindent = 10pt}
  \let\@listiii\@listii
  \let\@listiv\@listii
  \let\@listv\@listii
  \let\@listvi\@listii
  \makeatother
\ 
\begin{itemize}
\item[(I)]
$\lam_{21}=1$, $\lam_{11}=0$, and $\lam_{12}=1$.
In other words, 
$x_1=\mathbf{e}_2$ and $x_2=\mathbf{e}_1+\lam_{22}\mathbf{e}_2$.
\item[(II)]
$\lam_{22}=1$, $\lam_{11}=1$, and $\lam_{12}=0$.
In other words, 
$x_1=\mathbf{e}_1$ and $x_2=\lam_{21}\mathbf{e}_1+\mathbf{e}_2$.
\end{itemize}
}
%\noindent
%(Note that there may happen (I) and (II) simultaneously.)

\begin{cor}\label{cor N_1 ge N_2}
Suppose that $N_1 \ge N_2$.
We have the following.
\makeatletter
  \parsep   = 0pt
  \labelsep = .5pt
  \def\@listi{%
     \leftmargin = 20pt \rightmargin = 0pt
     \labelwidth\leftmargin \advance\labelwidth-\labelsep
     \topsep     = 0\baselineskip
     \partopsep  = 0pt \itemsep       = 0pt
     \itemindent = 0pt \listparindent = 10pt}
  \let\@listI\@listi
  \@listi
  \def\@listii{%
     \leftmargin = 20pt \rightmargin = 0pt
     \labelwidth\leftmargin \advance\labelwidth-\labelsep
     \topsep     = 0pt \partopsep     = 0pt \itemsep   = 0pt
     \itemindent = 0pt \listparindent = 10pt}
  \let\@listiii\@listii
  \let\@listiv\@listii
  \let\@listv\@listii
  \let\@listvi\@listii
  \makeatother
\ 
\begin{itemize}
\item[{\rm (i)}]
In the case of {\rm (I)}, 
$X_{\tilde{K}}$ is $\Zp[[S,T]]$-cyclic if and only if
%the condition that
%$Sx_1 \in (\pi, S)N \ox_{\Zp}\mathcal{O}$
%holds if and only if 
there exists some $x \in \mathcal{O}$ such that
\begin{eqnarray*}\label{criterion I}
\beta
=
(1+\lam_{22}\pi^k)\gam x.
\end{eqnarray*}
\item[{\rm (ii)}]
In the case of {\rm (II)}, 
$X_{\tilde{K}}$ is $\Zp[[S,T]]$-cyclic if and only if
%the condition
%that
%$Sx_1 \in (\pi, S)N \ox_{\Zp}\mathcal{O}$
%holds if and only if 
there exist some $f(S), g(S) \in \Lam$ such that
\begin{eqnarray*}\label{criterion II}
\alp
=
f(\alp)\lam_{21}\pi
,\ \ \beta
=
(\lam_{21}+\pi^k)(f(\beta)\pi+g(\beta)\gam\pi^k).
\end{eqnarray*}
\end{itemize}
\end{cor}

\begin{proof}
There is nothing to show for (ii)
by Lemma \ref{lem_criterion}.
We prove (i).
First, assume that $f(S),g(S) \in \Lam$ satisfy (\ref{criterion}).
Then $f(\alp)=0$.
Therefore, $f(S)$ is divisible $S-\alp$ by division lemma
(see \cite[Chapter VII \S 3.8 Proposition 5]{Bour comm alg}).
Put $x:=f(\beta)\pi/(\beta-\alp)+g(\beta)$.
Then we can easily check $x \in \mathcal{O}$ and
$\beta
=
(1+\lam_{22}\pi^k)\gam x
$.
Conversely, assume that there exist such $x \in \mathcal{O}$.
Put $f(S):=0$, $g(S):=x$
%Describe as $x=u+\pi y$ ($u \in \mathcal{O}^\x$, $y \in \mathcal{O}$) and put
%$f(S):=(S-\alp)y$,
%$g(S)=u$.
Then we can also easily check that they satisfy (\ref{criterion}).
\end{proof}

%%%%%%%%%%%%%%%%%%%%%%%%%%%%%%%%%%%%%%%%%%%%%%%%%%%%%

%%%%%%%%%%%%%%%%%%%%%%%%%%%%%%%%%%%%%%%%%%%%%%%%%%%%%

\subsection{The case where $k>0$}\label{k>0}

\begin{prop}\label{k>0 ord(gam)<min{ ord(alp),ord(beta)}}
If $k>0$ and 
$\ord (\gam)<\min\{ \ord(\alp),\ord(\beta)\}$,
then $X_{\tilde{K}}$ is $\Zp[[S,T]]$-cyclic.
\end{prop}

\begin{proof}
In this case, $A_K\ox_{\Zp}\mathcal{O} \isom \mathcal{O}/\gam \op \mathcal{O}/({\alp\beta}/{\gam})$ by Lemma \ref{A_K ox_{Zp} mathcal{O}}.
First, we suppose that $N_1=\ord(\gam)$, $N_2=\ord({\alp\beta}/{\gam})$.
Then $N_1<N_2$
(note that we cannot apply Corollary \ref{cor N_1 ge N_2} in this case).
Then, $\pi^{N_1}/\alp \notin \mathcal{O}$, so that $\pi \mid \lam_{11}$ by Lemma \ref{pi^{pi^N}x_1}.
Hence, $\lam_{12} \in \mathcal{O}^\x$.
We may assume that $\lam_{12}=1$ and also that $\lam_{21}=1$.
By Lemma \ref{pi^{pi^N}x_1}, we see that 
$\lam_{12}-\lam_{11}\gam/\alp$
must be divided by $\pi$,
so that we obtain
$\ord(\lam_{11})=\ord(\alp/\gam)$.
Put
\begin{eqnarray*}
&&
f(S)
:=
\lam_{11}\f{\alp}{\pi} 
\in \mathcal{O},
\\
&&
g(S)
:=
\f{\beta}{\beta-\alp}\cdot \f{\lam_{11}+\pi^k}{1+\lam_{22}\pi^k}-\f{\alp}{\beta-\alp}\lam_{11}
=
\f{(\beta-\alp)\lam_{11}+(\beta-\alp\lam_{11}\lam_{22})\pi^k}{(\beta-\alp)(1+\lam_{22}\pi^k)}.
\end{eqnarray*}
Then $g(S) \in \mathcal{O}$, since $1+\lam_{22}\pi^k \in \mathcal{O}^\x$ and $\ord(\beta-\alp)-k=\ord(\gam)<\min\{ \ord(\alp),\ord(\beta) \}$.
They satisfy (\ref{criterion}).

Second, we suppose that $N_1=\ord({\alp\beta}/{\gam})$, $N_2=\ord(\gam)$.
Then $N_1>N_2$ so that we can apply Corollary \ref{cor N_1 ge N_2}.
By Lemma \ref{pi^{pi^N}x_1}, we know that $\lam_{21}\gam/\alp \in \mathcal{O}$,
and hence $\ord(\lam_{21})>0$.
This implies that $\lam_{11},\lam_{22} \in \mathcal{O}^\x$, so we are reduced only to the case of (II).
Thus, we may assume that
$\lam_{11}=\lam_{22}=1$ and $\lam_{12}=0$.
Put 
$\del:=(\lam_{22}-\lam_{21}\gam/\alp)/\beta=(1-\lam_{21}\gam/\alp)/\beta$.
Then
\begin{eqnarray*}
\lam_{21}+\pi^k
=
\f{\alp}{\gam}(1-\beta\del)+\pi^k
=
\f{\beta}{\gam}(1-\alp\del).
\end{eqnarray*}
Note that $\ord(\del)\ge -\ord(\gam)>-\min\{ \ord(\alp),\ord(\beta)\}$ by Lemma \ref{pi^{pi^N}x_1},
so that both $1-\alp \del$ and $1-\beta \del$ are units.
Now, put
\begin{eqnarray*}
&&
f(S)
:=
\f{\gam}{\pi} (1-\beta\del)^{-1}
\in \mathcal{O},
\\
&&
g(S)
:=
\pi^{-k}\left( (1-\alp\del)^{-1}-(1-\beta\del)^{-1})\right)
=
\pi^{-k}\sum_{l=1}^\infty \left( \alp^l-\beta^l \right)\del^l.
\end{eqnarray*}
Here, the sum is convergent since $\ord(\alp\del)>0$ and $\ord(\beta\del)>0$.
%in fact, the order of $l$-th term is equal to or greater than 
%$\ord((\alp-\beta)\del)+(l-1)\left(\min\{ \ord(\alp),\ord(\beta)\}+\del\right)-k>0$.
Hence $g(S) \in \mathcal{O}$, and they satisfy the condition in Corollary \ref{cor N_1 ge N_2} (ii).
\end{proof}

%%%%%%%%%%%%%%%%%%%%%%%%%%%%%%%%%%%%%%%%%

The following proposition does not need the assumption that $k>0$.

\begin{prop}\label{ord(gam)>min{ ord(alp),ord(beta)}}
If $\ord(\gam)>\min\{ \ord(\alp),\ord(\beta)\}$,
then $X_{\tilde{K}}$ is not $\Zp[[S,T]]$-cyclic.
\end{prop}

\begin{proof}
Note that $\ord(\alp)=\ord(\beta)=N_1=N_2$ in this case, so that we can apply the criterion in Corollary \ref{cor N_1 ge N_2}.
Assume that $X_{\tilde{K}}$ is $\Zp[[S,T]]$-cyclic.
In the case of (I), the condition 
that
$(1+\lam_{22}
\pi^k
)^{-1}\beta/\gam \in \mathcal{O}$ 
contradicts the inequality $\ord(\gam)>\ord(\beta)$.
We consider 
the case of (II).
By $\ord(\gam)>\ord(\alp)$, 
we have
$
s:=1-\lam_{21}\gam/\alp \in \mathcal{O}^\x.
$
By the above assumption and Corollary \ref{cor N_1 ge N_2} (ii), there are some $f(S),g(S) \in \Lam$ such that
$\alp=f(\alp)\lam_{21}\pi$ and 
\begin{eqnarray*}
\beta\gam
&=&
(\lam_{21}\gam+(\beta-\alp))\left(f(\beta)\pi+g(\beta)\gam\pi^k\right)
\\
&=&
\beta\left(1-\f{\alp}{\beta}s\right)\left(f(\beta)\pi+g(\beta)\gam\pi^k\right).
\end{eqnarray*}
Now, $f(\beta)$ has a form $f(\beta)=f(\alp)+(\beta-\alp)A(\beta)$ with some $A(S) \in \Lam$
by division lemma,
%(see \cite[Chapter VII \S 3.8 Proposition 5]{Bour comm alg}), 
and $f(\alp)\pi=\alp/\lam_{21}=\gam(1-s)^{-1}$.
Therefore, 
\begin{eqnarray*}
1
&=&
\left(1-\f{\alp}{\beta}s\right)\left((1-s)^{-1}+A(\beta)\pi^{k+1}+g(\beta)\pi^k\right)
\\
&=&
\left(1-\f{\alp}{\beta}s\right)(1-s)^{-1}
+
\left(1-\f{\alp}{\beta}s\right)\left(A(\beta)\pi+g(\beta)\right)\pi^k.
\end{eqnarray*}
So we have
\begin{eqnarray*}
-s\f{1}{\lam_{21}}\cdot\f{\alp}{\beta}\pi^k
=
\left(1-\f{\alp}{\beta}s\right)\left(A(\beta)\pi+g(\beta)\right)\pi^k.
\end{eqnarray*}
The order of the left hand side is equal to or less than $k$, since $s \in \mathcal{O}^\x$.
On the other hand, the order of $1-{\alp}s/{\beta}$ is greater than $0$.
In fact, $s$ has a form $s=1+\pi t$ with some $t \in \mathcal{O}$, so that
$$
1-\f{\alp}{\beta}s
\equiv
1-\f{\alp}{\beta}
=
\f{\beta-\alp}{\beta}
\equiv 0
\mod \pi,
$$
since $\ord(\gam)>\ord(\beta)$.
This is a contradiction.
\end{proof}

\begin{prop}\label{k>0 ord(gam)=min{ ord(alp),ord(beta)}}
Suppose that $k>0$ and $\ord(\gam)=\min\{ \ord(\alp),\ord(\beta)\}$.
Then $X_{\tilde{K}}$ is $\Zp[[S,T]]$-cyclic if and only if $\lam_{21} \in \mathcal{O}^\x$.
\end{prop}

\begin{proof}
Note that $\ord(\gam)=\ord(\alp)=\ord(\beta)=N_1=N_2$, since $k>0$.
Therefore we can apply the criterion in Corollary \ref{cor N_1 ge N_2}.
In the case of (I), we see that $(1+\lam_{22}\pi^k)^{-1}\beta/\gam \in \mathcal{O}$, which implies 
that
$X_{\tilde{K}}$ is $\Zp[[S,T]]$-cyclic.
We consider 
the case of (II).
In the case where $\lam_{21} \in \mathcal{O}^\x$, we are reduced to 
the case of (I), 
so we may assume that $\pi \mid \lam_{21}$.
Then we have only to show that $X_{\tilde{K}}$ is not $\Zp[[S,T]]$-cyclic.
Assume the contrary.
Note that $s:=1-\lam_{21}\gam/\alp \in \mathcal{O}^\x$.
In the same way as Proposition \ref{ord(gam)>min{ ord(alp),ord(beta)}}, we get a contradiction,
since $(\beta-\alp)/\beta \equiv 0$ mod $\pi$ by $k>0$.
\end{proof}

%%%%%%%%%%%%%%%%%%%%%%%%%%%%%%%%%%%%%%%%%%%%%%%%%%%%%%%%%%

%%%%%%%%%%%%%%%%%%%%%%%%%%%%%%%%%%%%%%%%%%%%%%%%%%%%%%%%%%

\subsection{The case where $k=0$}\label{k=0}

Suppose that
$k=0$.
Then $X_{\Ki}\ox_{\Zp}\mathcal{O} \isom \Lam/(S-\alp)\op \Lam/(S-\beta)$.
We use the standard basis
$
\begin{bmatrix}
1
\\
0
\end{bmatrix}
$,
$
\begin{bmatrix}
0
\\
1
\end{bmatrix}
$, 
instead of 
$
\begin{bmatrix}
1
\\
1
\end{bmatrix}
$,
$
\begin{bmatrix}
0
\\
1
\end{bmatrix}
$.
If $\ord(\gam)>\min\{ \ord(\alp),\ord(\beta)\}$, then $X_{\tilde{K}}$ is not $\Zp[[S,T]]$-cyclic by Proposition \ref{ord(gam)>min{ ord(alp),ord(beta)}}.
So, in the following, we consider the case where
$\ord(\gam)=\min\{ \ord(\alp),\ord(\beta)\}$.
Moreover, we may assume that $\ord(\alp) \le \ord(\beta)$.
Express the generators $x_1$, $x_2$ as
\begin{eqnarray*}
x_1
=
\mu_{11}
\begin{bmatrix}
1
\\
0
\end{bmatrix}
+
\mu_{12}
\begin{bmatrix}
0
\\
1
\end{bmatrix}
,
\ \ 
x_2
=
\mu_{21}
\begin{bmatrix}
1
\\
0
\end{bmatrix}
+
\mu_{22}
\begin{bmatrix}
0
\\
1
\end{bmatrix}
\end{eqnarray*}
for some $\mu_{ij} \in \mathcal{O}$.
Then $Sx_1 \in (\pi,S)N \ox_{\Zp}\mathcal{O}$ if and only if
there exist some $f(S),g(S) \in \Lam$ such that
\begin{eqnarray}\label{criterion in k=0}
\begin{bmatrix}
\mu_{11}\alp
\\
\mu_{12}\beta
\end{bmatrix}
=
\begin{bmatrix}
\mu_{21}\pi \, f(\alp)+\mu_{21}\alp \, g(\alp)
\\
\mu_{22}\pi \, f(\beta)+\mu_{22}
\beta
\, g(\beta)
\end{bmatrix}.
\end{eqnarray}

%%%%%%%%%%%%%%%%%%%%%%%%%%%%%%%%%%%%%%%%%%%%%%%%%%%%%%%%%%%%%%%%%%%%%%%%%%%%%%%%%%%

\begin{prop}\label{k=0 ord(gam)=ord(alp)<=ord(beta) N_1<N_2}
Assume that $\ord(\beta-\alp)=\ord(\alp) \le \ord(\beta)$ and that $N_1<N_2$.
Then $X_{\tilde{K}}$ is $\Zp[[S,T]]$-cyclic if and only if 
$\mu_{21}$ is in $\mathcal{O}^\x$.
\end{prop}

\begin{proof}
In this case, 
$
A_K \ox_{\Zp} \mathcal{O} 
=
(\mathcal{O}/\pi^{N_1})x_1 \op (\mathcal{O}/\pi^{N_2})x_2 
\isom 
\mathcal{O}/\alp \op \mathcal{O}/\beta
$.
Hence $N_1=\ord(\alp)$, $N_2=\ord(\beta)$.
Therefore, we have $\ord(\mu_{12}) >0$.
In fact, if $\mu_{12} \in \mathcal{O}^\x$, then the order of $x_1$ in $A_K \ox_{\Zp} \mathcal{O}$ is $\#(\mathcal{O}/\pi^{N_2})$, which is a contradiction.
Thus, we may assume that $\mu_{11}=\mu_{22}=1$.
By (\ref{criterion in k=0}), $X_{\tilde{K}}$ is $\Zp[[S,T]]$-cyclic if and only if there exist some $f(S),g(S) \in \Lam$ such that 
$$
\alp=\mu_{21}(\pi f(\alp)+\alp g(\alp)),
\ \ 
\beta\mu_{12}=\pi f(\beta)+\beta g(\beta).
$$
Assume that $\ord(\mu_{21})>0$ and 
that
$X_{\tilde{K}}$ is $\Zp[[S,T]]$-cyclic.
By division lemma,
$f(\alp)=f(\beta)+(\alp-\beta)A(\alp)$ with some $A(S) \in \Lam$.
Then 
\begin{eqnarray*}
\ord(\alp)
&>&
\ord(\pi f(\beta)+\pi(\alp-\beta)A(\alp)+\alp g(\alp))
\\
&=&
\ord(\beta\mu_{12}-\beta g(\beta)+\pi(\alp-\beta)A(\alp)+\alp g(\alp)) \ge \ord (\alp).
\end{eqnarray*}
This is a contradiction.
Hence, if $X_{\tilde{K}}$ is $\Zp[[S,T]]$-cyclic, then $\mu_{21} \in \mathcal{O}^\x$.
Conversely, 
let $\mu_{21}$ be in $\mathcal{O}^\x$.
Put
$$
f(S):=\f{\alp\beta}{\pi(\beta-\alp)}\left( \mu_{21}^{-1}-\mu_{12} \right)
\in \mathcal{O},
\  \ 
g(S):=\f{\beta\mu_{12}-\alp\mu_{21}^{-1}}{\beta-\alp}
\in \mathcal{O}.
$$
Then we can easily check that
$\mu_{21}(\pi f(\alp)+\alp g(\alp))=\alp$ and that $\pi f(\beta)+\beta g(\beta)=\beta\mu_{12}$.
Hence $X_{\tilde{K}}$ is $\Zp[[S,T]]$-cyclic.
\end{proof}

\begin{prop}\label{k=0 ord(gam)=ord(alp)<=ord(beta) N_1>=N_2}
Assume that $\ord(\beta-\alp)=\ord(\alp) \le \ord(\beta)$ and that $N_1 \ge N_2$.
Then $X_{\tilde{K}}$ is $\Zp[[S,T]]$-cyclic if and only if 
it holds that $\mu_{21} \in \mathcal{O}^\x$ and $\ord(\mu_{22})=\ord(\beta)-\ord(\alp)$.
\end{prop}

\begin{proof}
In this case, we may change $x_1$ up to modulo $x_2 \Lam$
as we mentioned after the proof of Lemma \ref{an isom1}.
%in \S \ref{a system of generators}.
As in \S \ref{Methods}, we have only to consider the two cases:
{
\makeatletter
  \parsep   = 0pt
  \labelsep = .5pt
  \def\@listi{%
     \leftmargin = 20pt \rightmargin = 0pt
     \labelwidth\leftmargin \advance\labelwidth-\labelsep
     \topsep     = 0\baselineskip
     \partopsep  = 0pt \itemsep       = 0pt
     \itemindent = 0pt \listparindent = 10pt}
  \let\@listI\@listi
  \@listi
  \def\@listii{%
     \leftmargin = 20pt \rightmargin = 0pt
     \labelwidth\leftmargin \advance\labelwidth-\labelsep
     \topsep     = 0pt \partopsep     = 0pt \itemsep   = 0pt
     \itemindent = 0pt \listparindent = 10pt}
  \let\@listiii\@listii
  \let\@listiv\@listii
  \let\@listv\@listii
  \let\@listvi\@listii
  \makeatother
\ 
\begin{itemize}
\item[(I')]
$\mu_{21}=1$, $\mu_{11}=0$, $\mu_{12}=1$.
\item[(II')]
$\mu_{22}=1$, $\mu_{11}=1$, $\mu_{12}=0$.
\end{itemize}
}

First, we deal with the case of (I').
Then $\ord(\mu_{22}) \ge \ord(\beta)-\ord(\alp)$.
In fact, if $\ord(\alp)=\ord(\beta)$, then the inequality is trivial. 
And also, if $\ord(\alp)<\ord(\beta)$, then $N_2=\ord(\alp)$ by Lemma \ref{A_K ox_{Zp} mathcal{O}}, so that 
$
\pi^{N_2}x_2 \in S X_{\Ki} \ox_{\Zp} \mathcal{O}
=
\left\langle 
\begin{bmatrix}
\alp
\\
0
\end{bmatrix}
,
\begin{bmatrix}
0
\\
\beta
\end{bmatrix}
\right\rangle_{\mathcal{O}}
$.
This implies that $\mu_{22}\alp/\beta \in \mathcal{O}$.
By (\ref{criterion in k=0}),
$X_{\tilde{K}}$ is $\Zp[[S,T]]$-cyclic if and only if there exist some $f(S),g(S) \in \Lam$ such that
$$
0=\pi f(\alp)+\alp g(\alp),
\ \ 
\beta=\mu_{22}(\pi f(\beta)+\beta g(\beta)).
$$
Assume that $\ord(\mu_{22}) > \ord(\beta)-\ord(\alp)$ and $X_{\tilde{K}}$ is $\Zp[[S,T]]$-cyclic.
Since $\pi f(\beta)$ has a form $\pi f(\beta)=\pi f(\alp)+\pi(\beta-\alp)A(\beta)=-\alp g(\alp)+\pi(\beta-\alp)A(\beta)$ with some $A(S) \in \Lam$,
\begin{eqnarray*}
\ord(\beta)
&=&
\ord(\mu_{22})+\ord(-\alp g(\alp)+\pi(\beta-\alp)A(\beta)+\beta g(\beta))
\\
&>&
(\ord(\beta)-\ord(\alp))
+\ord(\alp).
\end{eqnarray*}
This is a contradiction.
Conversely, 
assume that 
$\ord(\mu_{22}) = \ord(\beta)-\ord(\alp)$.
Put
$$
f(S):=\f{-\alp\beta}{\mu_{22} \pi (\beta-\alp)}
\in \mathcal{O},
\  \ 
g(S):=\f{\beta}{\mu_{22}(\beta-\alp)}
\in \mathcal{O}.
$$
Then we can easily check
$\pi f(\alp)+\alp g(\alp)=0$ and $\mu_{22} (\pi f(\beta)+\beta g(\beta))=\beta$, and hence $X_{\tilde{K}}$ is $\Zp[[S,T]]$-cyclic.

Second, we deal with the case of (II').
By (\ref{criterion in k=0}), $X_{\tilde{K}}$ is $\Zp[[S,T]]$-cyclic if and only if there exist some $f(S),g(S) \in \Lam$ such that
$$
\alp=\mu_{21}(\pi f(\alp)+\alp g(\alp)),
\ \ 
0=\pi f(\beta)+\beta g(\beta).
$$
The similar argument in the case of (I') shows that $\ord(\alp) < \ord(\beta)$ does not occur.
In fact, the assumption $\ord(\alp) < \ord(\beta)$ yields $\mu_{22}\alp/\beta \in \mathcal{O}$, which is a contradiction since $\mu_{22}=1$.
Hence $\ord(\alp) 
=
 \ord(\beta)$.
Using this fact, in  the same way as the case of (I'), we can show that if $\ord(\mu_{21})>0$, then $X_{\tilde{K}}$ is not $\Zp[[S,T]]$-cyclic.
Conversely, if $\ord(\mu_{21})=0$, then we can prove that
$X_{\tilde{K}}$ 
is
 $\Zp[[S,T]]$-cyclic by taking
$$
f(S):=\f{\alp\beta}{\mu_{21} \pi (\beta-\alp)}
\in \mathcal{O},
\  \ 
g(S):=\f{-\alp}{\mu_{21}(\beta-\alp)}
\in \mathcal{O}.
$$
Then, the proof is completed.
\end{proof}

Summarizing all results in this section, we obtain the following.

\begin{thm}\label{main thm of classification lambda=2}
Let $p$ be an odd prime number, $K$ an imaginary quadratic field such that $p$ does not split.
Suppose that 
\makeatletter
  \parsep   = 0pt
  \labelsep = .5pt
  \def\@listi{%
     \leftmargin = 20pt \rightmargin = 0pt
     \labelwidth\leftmargin \advance\labelwidth-\labelsep
     \topsep     = 0\baselineskip
     \partopsep  = 0pt \itemsep       = 0pt
     \itemindent = 0pt \listparindent = 10pt}
  \let\@listI\@listi
  \@listi
  \def\@listii{%
     \leftmargin = 20pt \rightmargin = 0pt
     \labelwidth\leftmargin \advance\labelwidth-\labelsep
     \topsep     = 0pt \partopsep     = 0pt \itemsep   = 0pt
     \itemindent = 0pt \listparindent = 10pt}
  \let\@listiii\@listii
  \let\@listiv\@listii
  \let\@listv\@listii
  \let\@listvi\@listii
  \makeatother
\ 
\begin{itemize}
\item
$\dim_{\Fp}(A_K/p)=2$ and $\Gal(L_K \cap \tilde{K}/K)$ is a direct summand of $\Gal(L_K/K)$.
\item
The Iwasawa $\lam$-invariant of $\Ki/K$ is $2$.
\item
Let $\alp, \beta \in \overline{\Qp}$ be the roots of the distinguished polynomial generating the characteristic ideal 
of $X_{\Ki}$.
Then $\alp \neq \beta$.
\end{itemize}
%We may assume that $\ord(\alp) \le \ord(\beta)$.
Put $\mathcal{O}:=\Zp[\alp, \beta]$ and $l:=\min\{ \ord(\alp),\ord(\beta) \}$.
Let $x_2 \in X_{\Ki}$ be a preimage of a generator of $\Gal(L_K/L_K \cap \tilde{K})$.
Also, we denote by 
$
\begin{bmatrix}
\mu_{21}
\\
\mu_{22}
\end{bmatrix}
$
the image of $x_2 \ox 1$ under the injective map
$
X_{\Ki}\ox_{\Zp} \mathcal{O}
\inj 
\mathcal{O}[[S]]/(S-\alp) \op \mathcal{O}[[S]]/(S-\beta),
$
which is defined by {\rm (\ref{injection to elementary module})}.
Then, $X_{\tilde{K}}$ is $\Zp[[S,T]]$-cyclic if and only if one of the following holds:
%\begin{itemize}
%\item[{\rm (i)}]
%$k>0$, $\ord(\beta-\alp)-k< m$ 
%\item[{\rm (ii)}]
%$k>0$, $\ord(\beta-\alp)-k= m$, $\ord(\mu_{21})=0$
%\item[{\rm (iii)}]
%$k=0$, $\ord(\beta-\alp)= m$, $\ord(x_1)< \ord(x_2)$, $\ord(\mu_{21})=\ord(\mu_{22})=0$
%\item[{\rm (iv)}]
%$k=0$, $\ord(\beta-\alp)= m$, $\ord(x_1) \ge \ord(x_2)$, $\ord(\mu_{21})=0$, $\ord(\mu_{22})=\ord(\beta)-\ord(\alp)$,
%\end{itemize}
%where $m:=\min\{ \ord(\alp),\ord(\beta) \}$.
\vspace*{-6pt}
$$
\hspace*{-4pt}
\begin{array}{llllll}
{\rm (i)} & 
k>0, & \ord(\beta-\alp)-k< l, & &  
\hspace*{-15pt}
({\rm Proposition\ \ref{k>0 ord(gam)<min{ ord(alp),ord(beta)}}})
\\
{\rm (ii)} & 
k>0, & \ord(\beta-\alp)-k= l,&   \ord(\mu_{21})=0, & 
\hspace*{-15pt}
({\rm Proposition\ \ref{k>0 ord(gam)=min{ ord(alp),ord(beta)}}})
\\ 
{\rm (iii)} & 
k=0, & \ord(\beta-\alp)= l,\ n_1< n_2, & 
\ord(\mu_{21})=0,
&
\hspace*{-15pt}
({\rm Proposition\ \ref{k=0 ord(gam)=ord(alp)<=ord(beta) N_1<N_2}})
\\
{\rm (iv)} & 
k=0, & \ord(\beta-\alp)= l,\ n_1 \ge n_2, & 
\begin{cases}
\ord(\mu_{21})=0,
\\
\ord(\mu_{22})=\ord(\beta)-\ord(\alp),
\end{cases}
&
\hspace*{-15pt}
({\rm Proposition\ \ref{k=0 ord(gam)=ord(alp)<=ord(beta) N_1>=N_2}})
\end{array}
\vspace*{-6pt}
$$
where each $n_1$ and $n_2$ is defined by 
$p^{n_1}=\#\Gal(L_K \cap \tilde{K}/K)$ and $p^{n_2}=\# \Gal(L_K/L_K \cap \tilde{K})$, respectively.
%and
%$\Gal(L_K \cap \tilde{K}/K)\ox_{\Zp} \mathcal{O}=\mathcal{O}/\pi^{N_1}$.
\end{thm}

%\begin{rem}\label{no square root}
%It is expected that $\alp \neq \beta$.
%\end{rem}

%{\scriptsize 
{\small

\vspace*{10pt}
\noindent
Takashi MIURA,
\\
Department of Creative Engineering,
National Institute of Technology, Tsuruoka College,
\ 
104 Sawada, Inooka, Tsuruoka, Yamagata 997-8511, Japan.
\\
{\tt 
t-miura@tsuruoka-nct.ac.jp
}
\\[10pt]
Kazuaki MURAKAMI,
\\
Department of Mathematical Sciences, Graduate School of Science and Engineering, Keio University,
\ 
Hiyoshi, Kohoku-ku, Yokohama, Kanagawa 223-8522, Japan.
\\
{\tt 
murakami\_0410@z5.keio.jp
}
\\[10pt]
Rei OTSUKI,
\\
Department of Mathematics,
Keio University,
\ 
3-14-1 Hiyoshi, Kouhoku-ku, Yokohama 223-8522, Japan.
\\
{\tt 
ray\_otsuki@math.keio.ac.jp
}
\\[10pt]
Keiji OKANO,
\\
Department of Teacher Education,
\ 
3-8-1 Tahara, Tsuru-shi, Yamanashi 402-0054, Japan.
\\
{\tt 
okano@tsuru.ac.jp
}
\

}

\end{document}